\documentclass[preprint,12pt]{elsarticle}

\usepackage{amssymb}
\usepackage{amsmath}
\usepackage{multirow}%
\usepackage{algorithm}%
\usepackage{algorithmicx}%
\usepackage{algpseudocode}%
\usepackage{booktabs}%
\usepackage[linkcolor=blue,colorlinks=true]{hyperref}

\newtheorem{thm}{Theorem}
\newtheorem{theorem}[thm]{Theorem}

\newtheorem{proposition}[thm]{Proposition}
\newdefinition{rmk}{Remark}
\newdefinition{example}[rmk]{Example}
\newdefinition{definition}{Definition}
\newproof{proof}{Proof}
\newproof{pot}{Proof of Theorem \ref{thm2}}

\begin{document}

\begin{frontmatter}






\title{Approximation and exact penalization in simple bilevel variational problems}

\author[1]{Giancarlo Bigi}
\ead{giancarlo.bigi@unipi.it}

\author[2]{Riccardo Tomassini\corref{cor1}}
\ead{riccardo.tomassini@uniroma1.it}
\cortext[cor1]{Corresponding author}

\affiliation[1]{organization={Department of Computer Science, University of Pisa},
city={Pisa},
country={Italy}}

\affiliation[2]{organization={Department of Computer, Control and Management Engineering, Sapienza University of Rome},
city={Rome},
country={Italy}}

\begin{abstract}
A simple bilevel variational problem, where the lower level is a variational inequality while the upper level is an optimization problem, is studied. We consider an inexact version of the lower problem, which guarantees enough regularity to allow the exploitation of techniques of exact penalization. Moreover, cutting planes are used to approximate the Minty gap function of the lower level. Algorithms to solve the resulting inexact bilevel problem are devised relying on these techniques and approximations. Finally, their convergence is studied in detail by analyzing the effect of the given inexactness.
\end{abstract}



\begin{keyword}
Bilevel problem \sep Exact penalization  \sep Cutting planes \sep Variational inequalities



\end{keyword}

\end{frontmatter}



\section{Introduction}\label{sec1}

Variational problems play a fundamental role in optimization due to their extensive applications  (see, e.g. \cite{FePA97}), while bilevel problems have gained increasing attention in recent years (see, e.g. \cite{DeZe20}), since they are fit to model hierarchical decision making processes. There are many  possible formulations,  depending on the nature of the  upper and lower levels. In this paper, we focus on  bilevel problems where the lower level is non-parametric with respect to the upper level.
In particular, we study the following simple bilevel variational problem:
\begin{align}
&\min \{f(x) \mid x\in C, \ \langle G(x),y-x\rangle \ge 0, \ \forall y\in C\},\tag{OPVI} \label{OPVI}
\end{align}
where $C \subset \mathbb{R}^n$ is convex and compact, $f:C \to \mathbb{R}$ is convex, and $G:C \to\mathbb{R}^n$ is monotone and continuous.

Hierarchical optimization problems \cite{Ca05,So07,BeSa14,SaSh17,AnYo19,ShVuZe21,JiAbMoHa23,Giang-Tran2024} and, more recently, nested variational inequalities \cite{AnAn19,Lampariello2020,TAN2022106160,Lampariello2022,HiAn22,SaYu25} are some of the most studied problems in this area, often drawing inspiration from each other. Tikhonov-like approaches are a main tool for tackling these problems, and they arise from suitable penalization schemes for the bilevel problem. These techniques have been extensively studied for both optimization problems and variational inequalities.
Convergence results have been established for a variety of iterative methods, that leverage the first order optimality conditions of the penalized problem. For example, proximal point methods \cite{Ca05}, projected gradient methods \cite{So07,AnYo19}, and  conditional gradient methods \cite{Giang-Tran2024,JiAbMoHa23} have been used to find optimal solutions of bilevel optimization problems as the penalty parameter goes to infinity.
 Regularization methods \cite{Lampariello2020,Lampariello2022} have been successfully employed in  nested variational inequalities where the lower level is monotone, while extragradient methods \cite{TAN2022106160,AnAn19,HiAn22,SaYu25}, first introduced in \cite{Korpelevich1976TheEM} for single level variational inequalities,  have been used in the pseudo-monotone case. Other strategies include viscosity methods \cite{SaSh17,ShVuZe21}, inspired by \cite{XU2004279}, which reformulate the bilevel problem as a fixed point problem with variational constraints, and half-space optimization methods \cite{BeSa14}, where solutions are iteratively refined within dynamically updated half-spaces.
 
 Ad hoc methods for \eqref{OPVI} have been developed more recently under mild assumptions, in particular,   Tikhonov-like penalty approaches have been successfully adapted in \cite{KaFa21,KaSaYo23}. Furthermore, methods for simple bilevel optimization and nested variational inequalities can be employed as well, but only  when adequate assumptions allow for a reformulation of the problem. For instance, when $G$ is affine, the lower level  of \eqref{OPVI} can be reformulated as a convex optimization problem through gap functions, thereby enabling the use of techniques for simple bilevel optimization (see, e.g., \cite{So07,BeSa14,SaSh17}). In this case, exact penalization techniques have been exploited \cite{Bigi2022}  to prevent  the penalty parameter from going to infinity, which could lead to numerical issues.  
 This paper aims at going beyond the affine case by combining exact penalization with suitable cutting plane techniques for variational inequalities (see, e.g., \cite{Zangwill1971,Nguyen84,bigi2010successive}). Specifically, we introduce inexactness  into the lower level  and reformulate the bilevel problem using gap functions, obtaining a constrained single level problem. Inexactness allows bringing in regularity in the form of the Mangasarian-Fromovitz constraint qualification, so that exact penalization techniques can be exploited \cite{Ferris1991}. An analysis on  how inexactness propagates through different reformulations  of variational problems is carried out in \cite{Bigi2023}, while \cite{BeSte24} discusses how the inexactness can be leveraged to gain regularity in semi-infinite programming.  In our framework, inexactness comes at the cost of losing convexity in the lower level. Actually, this doesn't happen when the operator $G$ is affine and monotone since the gap function is convex. On the contrary, the convexity of  the gap function is not guaranteed when the operator is not affine. Therefore, we resort to the Minty gap function, which is  always convex but very challenging to evaluate, as the computation of any of its values requires the resolution of a non-convex minimization problem. To overcome this difficulty, we approximate the Minty gap function by using  cutting plane techniques. We prove that the solutions of the  penalized problem  coincide with the  solutions of the inexact bilevel problem for sufficiently large penalty parameters for any given degree of approximation. Moreover, we show that the required penalty parameter, while dependent on the choice of the cutting planes, remains uniformly bounded within the whole class of approximations. This allows devising a method that relies on exact penalization alongside of successive approximations of the Minty gap function without the penalty parameter to blow up.

The paper is structured as follows. Section 2.1 provides a mathematical introduction to \eqref{OPVI} and single level reformulations through gap functions, while Section 2.2 focuses on inexactness in the bilevel problem and the approximation of the Minty gap function in the lower level. Section 3 analyzes the class of penalized problems that arises from the approximations of the Minty gap function. In particular,  the exact penalty parameters are shown to be uniformly bounded within the entire class. In Section 4 we build algorithms and {analyze} their convergence, obtaining  theoretical bounds for the propagation of the inexactness. Finally, Section 5 reports some preliminary numerical tests 
in order to evaluate the sensitivity of the problem to the inexactness. 
In particular, the tests show that the actual final inexactness is generally meaningfully lower than the theoretical one. Some other tests have been run for an equilibrium selection problem, wherein the performances of the method are compared with other recent algorithms from literature.
\section{Bilevel variational problem}\label{sec2}

  \textcolor{black}{Throughout all the paper,} \textcolor{black}{let} the  following set of assumptions
\begin{itemize}

\item[-] \textbf{(A1)} the set $C \subset \mathbb{R}^n$ is convex and compact,
\item[-] \textbf{(A2)} the function $f:\mathbb{R}^n \rightarrow \mathbb{R}$ is convex,
\item[-] \textbf{(A3)} the map $G:\mathbb{R}^n  \rightarrow \mathbb{R}^n$ is continuous and monotone,
\end{itemize}
\textcolor{black}{hold, then} the bilevel problem \eqref{OPVI} has at least one optimal solution. The next subsection recalls some basic properties of variational inequalities and shows how \eqref{OPVI} can be reformulated as a constrained optimization problem using gap functions. 
 \subsection{General framework and mathematical tools}\label{subsec1}
We define the Stampacchia variational inequality 
\begin{align} \label{VI} \tag{VI}
    \text{Find} \quad x^* \in C: \langle G(x^*),y-x^*\rangle \ge 0 \quad \forall y\in C,  
\end{align}
which is closely related to the Minty variational inequality:
\begin{align} \label{Minty} \tag{MVI}
    \text{Find} \quad x^* \in C: \langle G(y),y-x^*\rangle \ge 0 \quad \forall y\in C.
\end{align}
 Under our assumptions, these two problems have the same set of solutions (see, e.g. \cite[Theorem 2.3.5]{FacchineiPang}). \textcolor{black}{ Therefore, \eqref{OPVI}
 can be recast also in another semi-infinite  format, i.e.
 \begin{align}
&\min \{f(x) \mid x\in C, \ \langle G(y),y-x\rangle \ge 0, \ \forall y\in C\}.\label{OPMVI}
\tag{OPMVI}
\end{align}}
\textcolor{black}{ Moreover, both the above variational inequalities}
can be reformulated as minimization problems through  gap functions, 
respectively  the Stampacchia gap function 
    \begin{align*}
        \psi^S(x):=\sup\{\langle G(x),x-y \rangle|y\in C\}
    \end{align*}
    for \eqref{VI},
    and the Minty gap function 
    \begin{align*}
        \psi^M(x):=\sup\{\langle G(y),x-y \rangle|y\in C\}
    \end{align*}
    for \eqref{Minty}.
 The Minty gap function  is always convex since it is  the pointwise supremum of a class of convex functions, while the Stampacchia gap function is not necessarily  convex when G is not affine. The following proposition summarizes  the aforementioned properties.
\begin{proposition}\label{reformulation} Assume \textbf{(A1)} and \textbf{(A3)} hold true, then  the following  are equivalent:
\begin{enumerate}
    \item $x^*\in C$ is a solution to \eqref{VI},
    \item $x^*\in C$ is a solution to \eqref{Minty},
    \item $x^* \in \arg \min \{\psi^S(x)\ | \ x\in C \}$ and $\psi^S(x^*)=0$,
    \item $x^* \in \arg \min \{\psi^M(x)\ | \ x\in C \}$ and $\psi^M(x^*)=0$.
\end{enumerate}
    
\end{proposition}

Therefore,
we can recast \eqref{OPVI} as the constrained optimization problem
\begin{align*}
&\min \{f(x)\ | \ x\in C,  \  \psi^S(x)\le 0\},
\end{align*}
or equivalently
\begin{align}\label{constrained}
&\min \{f(x)\ | \ x\in C,  \  \psi^M(x)\le 0\},
\end{align}
 since the gap functions have the same set of minimum points. 
 Lagrange multipliers are \textcolor{black}{a} useful tool for tackling  constrained problems, but their existence requires  adequate regularity conditions. \textcolor{black}{  In particular,  we say that a feasible point $x\in C$ satisfies the  Mangasarian-Fromovitz constraint qualification (MFCQ) for problem \eqref{constrained} if either $\psi^M(x)<0$ or  $\psi^M(x) = 0$ and 
\begin{align*}
    0\not \in \partial \psi^M(x)+N_C(x),
\end{align*}
where  $N_C(x)$ denotes the normal cone of $C$ at $x$ and  $\partial f(x)$  the subdifferential of $f$ at $x$. Notice that the above  Mangasarian-Fromovitz constraint qualification is equivalent to the condition for the semi-infinite   problem \eqref{OPMVI} introduced in \cite{JoTwWe92} (and further discussed in \cite{He92}) as the extended Mangasarian-Fromovitz constraint qualification.
Unfortunately, (MFCQ) doesn't hold for \eqref{constrained} at any feasible point. In fact, since   the feasible region coincides with the   set of solutions to \eqref{Minty}, every feasible point  satisfies both $\psi^M(x)=0$ and the  first order necessary optimality conditions of problem
$$\min \{\psi^M(x)\ | \ x\in C \},$$ in contradiction with the  constraint qualification.}
 To address this lack of regularity, we introduce inexactness into the lower level of \eqref{constrained}, which ensures that  (MFCQ) holds \textcolor{black}{at every feasible point} and consequently enables the use of exact penalization techniques.
 \subsection{Approximating the lower level} First of all, we discuss how  the set of solutions of the variational inequalities behaves when a positive inexactness  is added.  In particular, given $\varepsilon>0$
  consider the following inexact variational inequalities:
\begin{align}\label{eVI}\tag{$\varepsilon$-VI}
    \text{Find} \quad x^* \in C: \langle G(x^*),y-x^*\rangle \ge -\varepsilon \quad \forall y\in C
\end{align}
and
\begin{align}\label{eMVI}\tag{$\varepsilon$-MVI}
   \text{Find} \quad x^* \in C: \langle G(y),y-x^*\rangle \ge -\varepsilon \quad \forall y\in C.
\end{align}
While
   $\varepsilon=0$ entails that the sets of solutions coincide, 
this does not  generally hold true for $\varepsilon>0$. In fact, 
in contrast to Proposition \ref{reformulation},  the set of solutions  may be  different when $\varepsilon>0$ and just the following  inclusion always holds  true:
\begin{align}\label{contains}
\{x\in C\ |\ \psi^S(x)\le \varepsilon\} \subseteq \{x \in C\ | \ \psi^M(x)\le \varepsilon \}.
\end{align}
While the set of solutions of \eqref{eMVI} is always convex, the convexity of the set of solutions of \eqref{eVI} is not guaranteed beyond  the affine case, as shown in the next example. 
\begin{example}Consider the set $C=[0,1]^2$ and the operator $G(x_1,x_2)=(x_1^2 + x_2,x_2^2 - x_1)$ that is monotone on $C$.  The set of solutions to \eqref{eVI}, that is $\{x\in C\ | \ \psi^S(x)\le \varepsilon\}$,  is not convex for $\varepsilon={65}/{4096}$. For instance the points  $(0,0)$ and $({1}/{16},{1}/{4})$ are in the set, but their convex combination $({1}/{32},{1}/{8})$ is not.
 More generally, consider the   operator $G_{a,b}(x_1,x_2)=(ax_1^2+x_2,bx_2^2-x_1)$ that is  monotone on $C$ for all $a \ge 0$ and $b\ge0$. It can be shown that,  given any pair $a>0 ,b>0$, there exists $\varepsilon_{a,b}\le 1$ such that the set of solutions of \eqref{eVI} is not convex for all $ \varepsilon \in (0,\varepsilon_{a,b})$. In Figure \ref{fig1} some cases are illustrated: the set of solutions of \eqref{eVI}(light green) is not convex and it is contained in the set of  solutions of \eqref{eMVI}(dark green), which is always convex.
\end{example}
\begin{figure}[h!]
    \centering
    \includegraphics[width=0.8\linewidth]{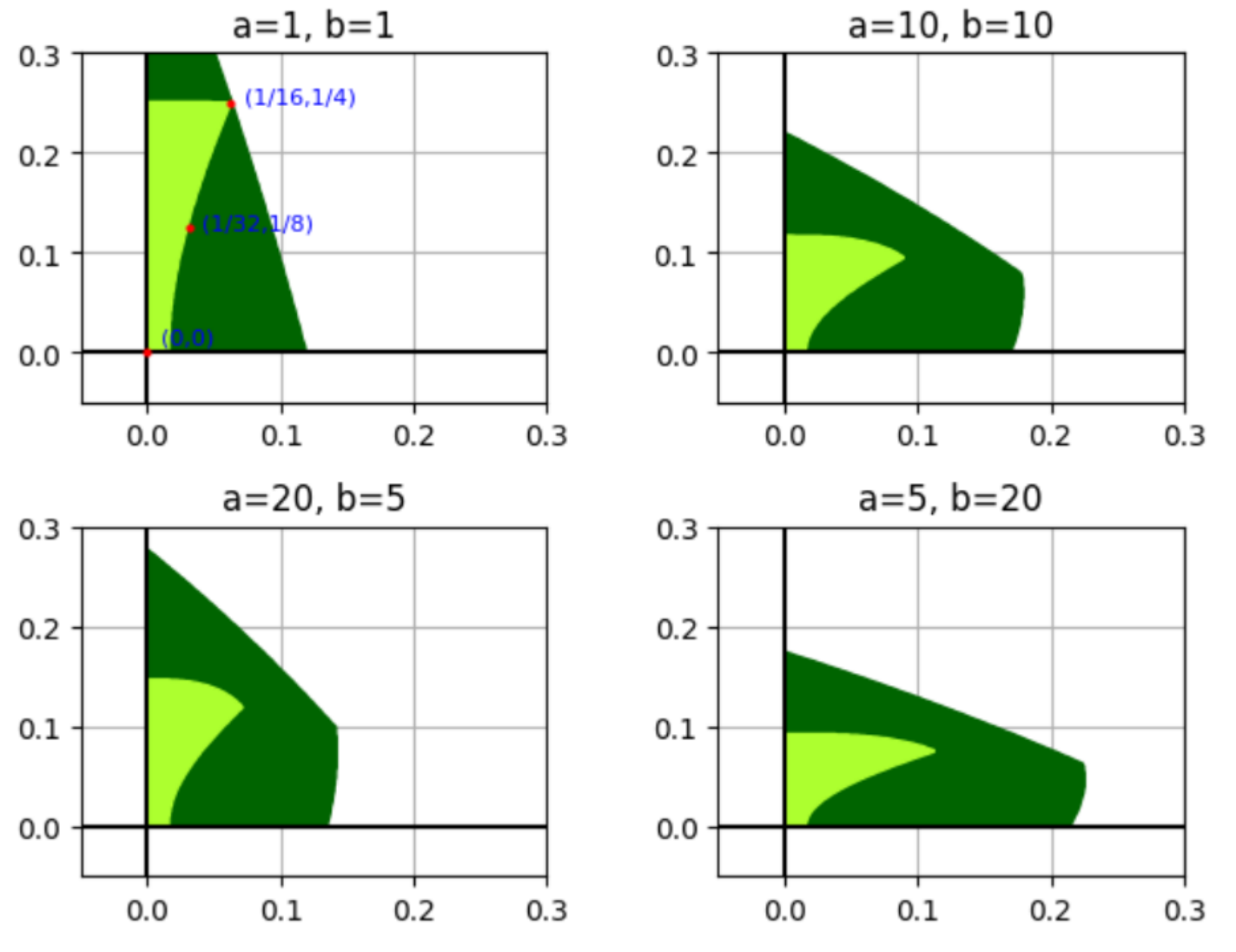}
    \caption{$\varepsilon$-solutions for different $a$ and $b$ for the operator $G_{a,b}$ on the set $[0,1]^2$ }
    \label{fig1}
\end{figure}

Now, consider the Stampacchia inexact version  of \eqref{OPVI}:
\begin{align}
&\min \{f(x)\ | \ x\in C,  \  \psi^S(x)\le \varepsilon \}.\tag{$\varepsilon$-OPVI} \label{OPVIe}
\end{align}
When $G$ is affine, exact penalization techniques have been employed \cite{Bigi2022} by leveraging the regularity gained as a consequence  of the  inexactness. 
 To extend this type of  approach beyond the affine  case, we  can not  rely on  the convexity for the Stampacchia gap function any more. Since the Minty gap function   is always  convex and the set of solutions of \eqref{eMVI} always contains  the solutions of \eqref{eVI}, we  focus on the inexact Minty bilevel problem:
\begin{align}
&\min \{f(x)\ | \ x\in C,  \  \psi^M(x)\le \varepsilon \}\tag{$\varepsilon$-OPMVI}. \label{OPMYe}
\end{align}
The optimal solutions of the above problem  behave well in respect to the original problem, meaning that  they \textcolor{black}{converge} to the optimal solutions of  \eqref{OPVI} as $\varepsilon$ goes to zero. 
\begin{proposition}
   Assume \textbf{(A1)-(A3)} hold true and consider a sequence $\{x^*_{\varepsilon_k} \}_{k}$ with $\varepsilon_k \downarrow0$ where $x^*_{\varepsilon_k}$ is an optimal solution to ($\varepsilon_k\text{-OPMVI}$). Then, every cluster point $\bar{x}$ of the sequence is an optimal solution to \eqref{OPVI}.
\end{proposition}
\begin{proof}
	Without loss of generality, we can assume   $x^*_{\varepsilon_k} \rightarrow \bar{x}$, for $k\rightarrow +\infty$. Taking the limit, we get
	$$\psi^M(\bar{x})=\lim_{k \rightarrow +\infty}\psi^M(x^*_{\varepsilon_k}) - \varepsilon_k \le 0,$$
	which implies that  $\bar{x}$ is feasible for \eqref{constrained},  \textcolor{black}{ hence for \eqref{OPVI} by Proposition \ref{reformulation}}. If $\bar{x}$ \textcolor{black}{was not} a minimum point for  $f$, then \textcolor{black}{there} would exists $\tilde{x}\in C$ such that  $f(\tilde{x})<f(\bar{x})$ and $\psi^S(\tilde{x})\le 0$. In this case,  $\tilde{x}$ would be feasible for  ($\varepsilon_k\text{-OPMVI}$), thus obtaining:
	$$f(x^*_{\varepsilon_k}) \le f(\tilde{x}) <f(\bar{x}).$$
	Taking  the limit we obtain the contradiction $f(\bar{x})<f(\bar{x})$.
\end{proof}

  While switching to the Minty gap function takes care of the possible non-convexity of the lower level, a new issue arises. In fact, computing  the Minty gap function at a point is  challenging when $G$ is not affine, since it requires solving a non-concave maximization problem. To address this issue, we  approximate  the Minty gap function through cutting planes techniques.  In particular, given any subset $B \subseteq C$, we define its corresponding  approximated Minty gap function
\begin{align*}
    \psi^M_{B}(x):=\sup\{\langle G(y),x-y \rangle \ | \ y \in B\}
\end{align*}
which is still convex and lies below $\psi^M$ at any point.
 Notice it  may also take negative values at points  $x\notin B$. We further relax  \eqref{OPMYe}  by replacing $\psi^M$ with $\psi^M_B$ , that is we consider the constrained problem:
\begin{align}
&\min \{f(x)\ | \ x\in C,  \  \psi_B^M(x)\le  \varepsilon\}. \label{OPMYeB}\tag{$(\varepsilon,B)$-OPMVI}
\end{align}
Notice that  the explicit constraint in the above problem amounts to finite number of inequalities when $B$ is finite.
\textcolor{black}{Moreover, we can find the  optimal solution to the original problem by considering a sequence of sets that approaches  a dense subset of $C$. Given a sequence $\{B_k\}_k$ of subsets of $C$ that is non decreasing with respect to the inclusion, we  define 
$$B:=\cup_{k\in \mathbb{N}} B_k,$$
so that the sequence    converges  to $B$   with respect to the Hausdorff distance~\cite{hausdorff1991set}.}

\begin{proposition}
   \textcolor{black}{Assume \textbf{(A1)-(A3)} hold true and consider a sequence $\{\varepsilon_k\}_{k}$ with  $\varepsilon_k \downarrow0$ and a sequence $\{B_k\}_{k}$ of non decreasing sets that converges to some $B$ dense in $C$. Let  $x^*_{\varepsilon_k}$  be an optimal solution to ($(\varepsilon_k,B_k)\text{-OPMVI}$). Then, every cluster point $\bar{x}$ of the sequence $\{x^*_{\varepsilon_k} \}_{k}$ is an optimal solution to \eqref{OPVI}.}
\end{proposition}

\begin{proof}
	\textcolor{black}{Without loss of generality, we can assume   $x^*_{\varepsilon_k} \rightarrow \bar{x}$. Now, take any $x \in C$. The sequence $\{\psi^M_{B_k}(x)\}_k$ is non decreasing and bounded from above by $ \psi^M(x)$. Therefore, it  admits a limit that is   exactly $\psi^M_B(x)$ since $B_k$ converges to $B$. Thus, we get $$\psi^M_{B_k}(x)\uparrow\psi^M(x),$$ as  the equality $\psi^M_B(x)=\psi^M(x)$ holds since   $B$ is dense in $C$.
    Moreover,~Theorem 7.13 \cite{Rudin1976} guarantees that the convergence is uniform thanks to the continuity of  the functions  $\psi_{B_k}^M$ and  $\psi^M$. 
    Thus, taking the limit, we get
    $$\psi^M(\bar{x})=\lim_{k \rightarrow +\infty}\psi_{B_k}^M(x^*_{\varepsilon_k}) - \varepsilon_k \le 0,$$
	that implies that  $\bar{x}$ is feasible for \eqref{constrained},  hence  for \eqref{OPVI} by Proposition \ref{reformulation}.  Ab absurdo, suppose that  $\bar{x}$  is not a minimum point of \eqref{OPVI}.  Then, there exists $\tilde{x}\in C$ such that  $f(\tilde{x})<f(\bar{x})$ and $\psi^S(\tilde{x})\le 0$.
    Then,  $\tilde{x}$ is feasible for  ($(\varepsilon_k,B_k)\text{-OPMVI}$) leading to the inequalities
	$$f(x^*_{\varepsilon_k}) \le f(\tilde{x}) <f(\bar{x}).$$
	Taking  the limit we obtain the contradiction $f(\bar{x})<f(\bar{x})$.}
\end{proof}

\textcolor{black}{The  inexactness guarantees that any solution $x^*\in C$ to the variational inequality  \eqref{Minty} is also a  Slater point  for problem \eqref{OPMYeB}, that is  $\psi^M_B(x^*)< \varepsilon $.  Leveraging the convexity of $\psi^M_B$, the existence of a Slater point implies that    (MFCQ) holds for problem \eqref{OPMYeB} at every feasible point.}
\textcolor{black}{Indeed,  convexity guarantees the equivalence 
$$x\in \arg\min \{ \psi^M_B(x) \ | \  x \in C \} \iff 0\in \partial\psi^M_B(x)+N_C(x), $$
which, in conjunction with the existence of a Slater point, implies that  
\begin{align}\label{eMFCQ}
    \{x \in C \ | \ 0 \in \partial\psi^M_B(x)+N_C(x) \} \cap\{x \in C \ | \ \psi^M_B(x)\ge \varepsilon \}=\emptyset
\end{align}
so that the qualification condition   is satisfied even beyond the feasible region.
Moreover, convexity entails also that  \eqref{eMFCQ} is  equivalent to    the global error bound over  $C$ as given in \cite[Definition 7.1]{Io16} thanks to  \cite[Theorem  7.2]{Io16}.}
\textcolor{black}{ Therefore, the  inexactness pays off as it guarantees enough regularity to employ exact penalty methods \cite{Ye12}.}

 \textcolor{black}{To this aim, for any given $\varepsilon>0$ and $B \subseteq C$, we define the problem
\begin{align*}
&\min\{f(x) + \rho (\psi^M_{B}(x)-\varepsilon)^+ \tag{$P(\rho,\varepsilon,B)$}\ | \   \ x\in C\}, \label{Pen}
\end{align*}
where the  positive part of the constraint   is penalized.} \textcolor{black}{In the next section, we prove that   all the optimal solutions to 
\eqref{Pen} 
  are also optimal solutions to \eqref{OPMYeB} when the penalty parameter is sufficiently large. Moreover, we show that this  threshold is uniformly bounded  over all the possible choices of the set $B$.}

\section{Uniform exact penalization}
\label{sec3}

 Before presenting the main results, we need to recall some properties regarding set valued maps (see \cite{rockafellar2009variational} for proofs and more details). \begin{definition}
    Given a set $C \subseteq \mathbb{R}^n$,  the set-valued map $F :C \rightrightarrows \mathbb{R}^n$ is closed  if:
    \begin{align*}
        x_i \rightarrow x_\infty, \hspace{0.2cm} y_i \in F(x_i), \hspace{0.2cm} y_i\rightarrow y_{ \infty}\implies y_\infty \in F(x_\infty).
    \end{align*}
\end{definition}

In our framework, 
 $N_C: C\rightrightarrows \mathbb{R}^n$ ,  $\partial f:C \rightrightarrows \mathbb{R}^n $ and $\partial \psi^M:C \rightrightarrows \mathbb{R}^n $ are closed set-valued maps.
Also, $f$ and $\psi^M_B$ are both real-valued convex  functions, so they are locally Lipschitz continuous on $\mathbb{R}^n$. Since $C$ is compact, local Lipschitz continuity guarantees that  their subdifferentials are bounded on $C$, i.e. there exist positive $R_f$ and $R_B$ such that any $x$ in $C$ satisfies:
 \begin{align*}
      \partial f(x)\subseteq B(0,R_f) \hspace{0.2cm} \text{and} \hspace{0.2cm} \partial \psi^M_B(x)\subseteq B(0,R_B). 
 \end{align*}
 Moreover, we have the following explicit description of the subdifferential of the approximated Minty gap function 
 \begin{align*}
            \partial \psi^M_{B}(x) =\text{conv}\{G(y)\ | \ y \in \arg\max_{z\in \overline{B}} \langle G(z),x-z\rangle \},
        \end{align*} where the operator conv denotes the convex hull of the set  and  $\overline{B} $ denotes the closure of  $B$.
As a consequence,  there exists a bound on the  subdifferentials that is independent from the choice of $B$:
 \begin{align*}
       \partial \psi^M_B(x)\subseteq B(0,\bar{G}), \hspace{0.2cm}\forall x\in C, \forall B\subseteq C, 
 \end{align*}
where $\bar{G}:=\max \{\ \|G(y)\| \ | \ y\in C\}$.
 In the next proposition we exploit both closedness and boundedness to show that the solutions of  each penalized problem  are also solutions to inexact bilevel problem \eqref{OPMYeB} whenever  the penalty parameter is sufficiently large.
\begin{proposition}\label{proprhoB}
   Assume   \textbf{(A1)-(A3)} hold true. Then, there exists  ${\rho}_B>0$ such that any   solution to \eqref{Pen} is also a solution to \eqref{OPMYeB} whenever $\rho \ge \rho_B$.
\end{proposition}
\begin{proof}
      By contradiction, assume there exist sequences $\{\rho_k\}_{k}$ and $\{x_k\}_{k}$ such that $\rho_k \rightarrow +\infty$ and $x_k$ solves ($P(\rho_k,\varepsilon,B)$) but not \eqref{OPMYeB}. Therefore, $x_k$ is not feasible for the latter problem, so that $\psi^M_{B}(x_k) >\varepsilon$ and hence $\partial(\psi^{M}_{B}(x_k)-\varepsilon)^+=\partial\psi^{M}_{B}(x_k)$. Since ($P(\rho_k,\varepsilon,B)$) is a convex  problem, the optimality conditions for $x_k$ read
    \begin{align*}
	    0\in \partial f(x_k) + \rho_k\partial\psi^{M}_{B}(x_k) + N_C(x_k),
    \end{align*}
    or equivalently
    \begin{align*}
    0 \in \frac{1}{\rho_k}\partial f(x_k) + \partial\psi^{M}_{B}(x_k) + N_C(x_k).
	\end{align*}
\textcolor{black}{Thanks to the compactness of $C$, there exists a point $\bar{x}\in C$ and a subsequence $\{x_{k_m}\}_{m}$ such that $x_{k_m} \rightarrow \bar{x}$.} 
Since the subdifferentials of $f$ are bounded on $C$, we have
$$\dfrac{1}{{\rho}_{k_m}}\partial f(x_{k_m}) \subseteq\dfrac{1}{{\rho}_{k_m}}B(0,R_f)$$
for some $R_f>0$. Since   $ \partial \psi^M_{B}$ and  $ N_C$ are   closed maps and  $ \partial \psi^M_{B}$  is also bounded on $C$, taking the limit provides: 
 \begin{align*}
        0 \in \partial \psi^M_{B}(\bar{x})+ N_C(\bar{x}).
\end{align*}
 Finally,  \textcolor{black}{$\psi^{M}_{B}(\bar{x})\ge\varepsilon$  holds by the continuity of $\psi^M_B$, in contradiction with \eqref{eMFCQ}.} 
	
\end{proof}

From now onwards
 $\rho_B$ will denote the minimum  of  the exact penalty parameters in the above proposition.  In principle it might  exist a particular sequence of sets $\{B_k\}_{k}$ such that the  sequence of exact penalty parameters $\rho_{B_k}$ goes to infinity. Actually, this is not the case since  there exists a uniform bound on the set of exact penalty parameters. To show this,  we need a result on the uniform convergence for  an arbitrary sequence of approximated Minty gap functions.

\begin{proposition}\label{Uniform} Assume \textbf{(A1)} and \textbf{(A3)} hold true and take a  sequence of subsets $\{B_k\}_{k\in \mathbb{N}}$  of $C$.  Define the sets:
\begin{align*}
    \hat{B}_k:=\overline{\left(\bigcup_{m \ge k}B_m\right)} \hspace{0.5cm} \text{and} \hspace{0.5cm} \hat{B}: = \bigcap_{k= 1}^{+\infty}\hat{B}_k.
\end{align*}
Any $x\in C$ satisfies
\begin{align*}
    \sup_{m\ge k}\psi^M_{B_m}(x)&=\psi^M_{\hat{B}_k}(x),
\end{align*}
and
\begin{align*}
\lim_{k\rightarrow + \infty}\psi^M_{\hat{B}_k}(x) =\psi^M_{\hat{B}}(x),
   \end{align*}
  where the convergence is uniform on $C$.
	\end{proposition}
 \begin{proof}
 Notice that  $\hat{B}$ is not empty since  the family of sets $\{\hat{B}_k\}_{k\in \mathbb{N}}$ is non increasing, meaning that $\hat{B}_{h}\subseteq \hat{B}_k$ holds whenever $h \ge k$, and therefore it  has the finite intersection property.
 The inequality $\sup_{m\ge k}\psi^M_{B_m}(x)\le\psi^M_{\hat{B}_k}(x)$ follows from  $B_m\subseteq \hat{B}_k$ for all  $m \ge k$. We  show that the opposite inequality also holds true. We take   $\bar{y}$ such that:
	$$\psi^M_{\hat{B}_k}(x)=\langle G(\bar{y}),x- \bar{y}\rangle.$$
Since $\bar{y} \in \hat{B}_k$, there exists a sequence of indices $\{m_j\}_j$ with $m_j \ge k$,  such that $y_{m_j}\in B_{m_j}$ and $y_{m_j}\rightarrow \bar{y}$. Then,  the following inequalities hold true:
\begin{align*}
    \sup_{m\ge k}\psi^M_{B_m}(x)\ge \psi^M_{B_{m_j}}(x) \ge \langle G(y_{m_j}),x- y_{m_j}\rangle.
\end{align*}
Taking the the limit, we get the desired inequality 
\begin{align*}
    \sup_{m\ge k}\psi^M_{B_m}(x)\ge \langle G(\bar{y}),x- \bar{y}\rangle=\psi^M_{\hat{B}_k}(x),
\end{align*}
thanks to the continuity of $G$.

Given $x \in C$, the inequality $\psi^M_{\hat{B}_k}(x) 
\ge \psi^M_{\hat{B}}(x)$ holds true since $\hat{B} \subseteq \hat{B}_k$ and moreover the left-hand side is a non increasing sequence. Therefore, this sequence of values has a limit. Consider any $y_k \in \hat{B}_k$ such that $\psi^M_{\hat{B}_k}(x)=\langle G(y_k),x- y_k\rangle$. Without any loss of generality, we can assume that $y_k \rightarrow \hat{y}$ for some $\hat{y} \in C$. Since the sequence of sets is non increasing, then $\hat{y}$ belongs to their intersection, i.e. $\hat{y} \in \hat{B}$. Thanks to the continuity of $G$, we get:  
$$\psi^M_{\hat{B}}(x)\le \lim_{k \rightarrow +\infty} \psi^M_{\hat{B}_k}(x) =\lim_{k \rightarrow +\infty} \langle G(y_{k}),x- y_{k}\rangle= \langle G(\hat{y}),x- \hat{y}\rangle\le \psi^M_{\hat{B}}(x).$$
 
 The uniform convergence follows from Theorem 7.13 \cite{Rudin1976}
 . In fact, the sequence of continuous functions $\{\psi^M_{\hat{B}_k}\}_{k \in \mathbb{N}}$ is non increasing on the compact set $C$, and  its pointwise limit $\psi^M_{\hat{B}}$ is also continuous. Therefore,   the convergence is uniform.
\end{proof}

The next theorem states that there exists a uniform bound for the exact penalty parameters. The main idea is to generalize the proof of Proposition \ref{proprhoB} by considering a sequence of problems \eqref{OPMYeB}. Proposition \ref{Uniform} is a key tool for the proof as it guarantees the existence of a suitable ``limit'' set  
and the corresponding  approximated problem 
 for which  {\eqref{eMFCQ}}  still holds.

\begin{thm}\label{Uniform_exact_penalization}
    Assume \textbf{(A1)-(A3)} hold true. Then, $\rho_B$ is bounded uniformly over all possible subsets $B \subseteq C$, that is
    \begin{align*}
        \sup_{B \subseteq C} \rho_B < +\infty.
    \end{align*}
        
\end{thm}
\begin{proof} By contradiction, suppose  there exists a sequence of sets  $\{B_k\}_{k \in  \mathbb{N}}$ such that $\rho_{B_k} \uparrow +\infty$. Consider $\rho_k=\rho_{B_k}/2$, hence there exists $x_k\in C$ such that  $\psi^M_{B_{k}}(x_k) >\varepsilon$, while
\begin{align*}
    x_k\in \arg\min \{ f(x)+ \rho_k (\psi^M_{B_{k}} (x)-\varepsilon)^+\ | \ x \in C\}.
\end{align*}
The optimality conditions guarantee the existence of $v_k\in \partial \psi^M_{B_{k}}(x_{k})$ and $u_k\in N_C(x_k)$ such that:
 \begin{align}\label{FVU}
        0=\frac{1}{\rho_k}\partial f(x_{k}) + v_k + u_k.
    \end{align}
Since all the subdifferentials are contained in a unique ball, we can assume $v_k \rightarrow v_\infty$ as well as $x_k \rightarrow x_\infty$. Then, equation \eqref{FVU} implies $u_k \rightarrow -v_\infty$  and the closedness of the normal cone guarantees $-v_\infty\in N_C(x_\infty)$. The subgradient inequality provides
$$\psi^M_{B_k}(y)\ge \psi^M_{B_{k}}(x_{k}) + \langle v_{k}, y - x_{k} \rangle.$$
Thanks to Proposition \ref{Uniform}, taking the supremum on both sides gives:
\begin{align}\label{I:sup}
	\psi^M_{\hat{B}_{k}}(y)&\ge \sup_{h\ge k}(\psi^M_{B_{h}}(x_{h}) + \langle v_{h}, y - x_{h} \rangle)\notag\\
	&\ge  \sup_{h\ge k}\psi^M_{B_{h}}(x_{h}) + \inf_{h\ge k}\langle v_{h}, y - x_{h} \rangle.
\end{align}
We also have:
\begin{align*}
		\psi^M_{B_{h}}(x_\infty) &= \sup_{y \in B_{h}} \langle G(y),x_\infty -y \rangle  \\
		& =\sup_{y \in B_{h}} \big( \langle G(y),x_\infty -x_{h} \rangle + \langle G(y),x_{h}-y \rangle\big)\\
		&\le \bar{G}||x_\infty-x_{h}|| + \psi^M_{B_{h}}(x_{h}),
\end{align*}
where the final inequality follows from the Cauchy-Schwarz inequality and the definition of the approximated Minty gap function. Plugging this latter relation into \eqref{I:sup}, we get:
\begin{align*}
	\psi^M_{\hat{B}_{k}}(y)&\ge  \sup_{h\ge k}\big(\psi^M_{B_{h}}(x_\infty) - \bar{G}||x_\infty-x_{h}||\big) +\inf_{h\ge k} \langle v_{h}, y - x_{h} \rangle\\
    & \ge \sup_{h\ge k}\psi^M_{B_{h}}(x_\infty) - \sup_{h \ge k} \bar{G}||x_\infty-x_{h}||+\inf_{h\ge k} \langle v_{h}, y - x_{h} \rangle.
	\end{align*}
Taking the limit, Proposition \ref{Uniform} guarantees
 \begin{align*}
        \psi^M_{\hat{B}}(y)\ge \psi^M_{\hat{B}}(x_\infty) + \langle v_\infty,y -x_\infty\rangle,
    \end{align*}
which implies $v_\infty \in \partial \psi^M_{\hat{B}}(x_\infty)$. Therefore, we get:
\begin{align*}
         0\in \partial \psi^M_{\hat{B}}(x_\infty) + N_C(x_\infty).
     \end{align*}
Finally, the uniform \textcolor{black}{convergence} in Proposition \ref{Uniform} guarantees
\begin{align*}
        \psi^M_{\hat{B}}(x_\infty) = \lim_{k\rightarrow +\infty}\psi^M_{\hat{B}_{k}}(x_k)\ge  \varepsilon,
    \end{align*}
\textcolor{black}{which gives the contradiction, as \eqref{eMFCQ} holds for every choice of $B \subseteq C$.}

\end{proof}
\section{Algorithms}\label{sec4}
In this section we present three algorithms, employing techniques of  cutting planes (\cite{Zangwill1971} and \cite{Nguyen84}) to approximate the value of the Minty gap function. The first algorithm   requires 
\textcolor{black}{the exact evaluation of} the Minty gap function  at a given point and finds an optimal solution to \eqref{OPMYe}. \textcolor{black}{Since evaluating this function is challenging when the operator $G$ is not affine, we present  two algorithms  that utilize both the Stampacchia gap function and  a computable approximation of the Minty gap function. They    require that  the operator $G$ is Lipschitz, but they are able to compute only an approximated  optimal solution to \eqref{OPMYe}.} 

\begin{figure}[h]
    \centering
    \includegraphics[width=1\linewidth]{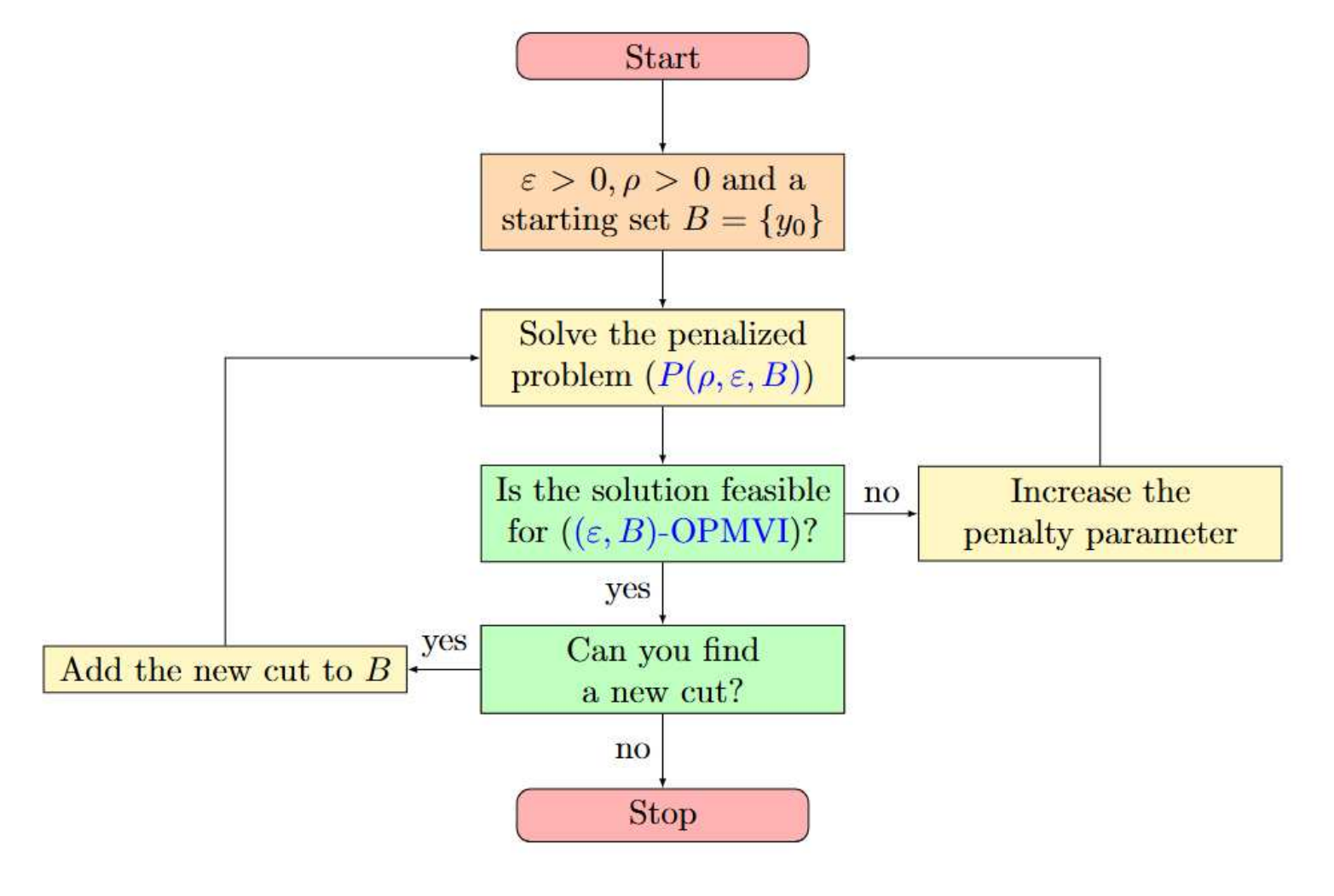}
    \caption{Flowchart of the method}
    \label{fig2}
\end{figure} 

All the  algorithms  have been devised using the following general scheme.
At each iteration we compute an optimal solution  to the penalized problem \eqref{Pen} and check whether the computed solution is feasible for \eqref{OPMYeB} or not. If it isn't, we increase the penalty parameter. This  can happen only a finite number of times thanks to  Theorem \ref{Uniform_exact_penalization}, since  the penalty parameter stops growing as soon as it  exceeds the uniform bound. Otherwise, if the optimal solution is feasible for \eqref{OPMYeB}, we exclude it from successive iterations by adding a new element to $B$. In fact, the more $B$ grows, the closer  the approximated Minty gap  function gets to the  Minty gap function. Geometrically, this corresponds to  increasing the number of  half-spaces used to approximate  the  feasible region of  \eqref{OPMYe}. Therefore, we will refer to  the elements of $B$ as cuts.

 More specifically, we start with a random cut $B_0=\{y_0\}$ and add new cuts as the algorithm goes on. Given a set of cuts $B_k=\{y_0,\dots,y_k\}$, we define the function:
\begin{align*}
&\psi^M_{k}(x):=\max\{\langle G(y_i),x-y_i \rangle \ | \ {i \in \{0,\dots,k\}} \}
\end{align*}
and in the limit case:
\begin{align*}
&\psi^M_{\infty}(x):=\sup\{\langle G(y_i),x-y_i\rangle \ | \ {i \in \mathbb{N}}\}.
\end{align*}

The way we choose new cuts determines the converging properties of the algorithm. Using the value  of the Minty gap function, we compute cuts that  lead the algorithm to the optimal solutions to \eqref{OPMYe}.
\begin{algorithm}[h]
\caption{}\label{alg1}
\begin{algorithmic}[1]
\State \textbf{input:} $y_0 \in  C$, $\rho_0>0$, $\sigma>1$, $\delta_k>0$ and $ \delta_k\downarrow0$.
\For{$k=0,1,\dots$}
\State $z \in \arg\min\{f(x) + \rho_k (\psi_k^M(x)-\varepsilon)^+\ | \ x \in C\}$
\While{$\psi^{M}_{k} (z)>\varepsilon$}
    \State{$ \rho_{k}=\sigma \rho_k$
    \State$z \in \arg\min\{f(x) + \rho_k (\psi_k^M(x)-\varepsilon)^+\ | \ x \in C\}$}
\EndWhile
\State $x_k=z$
\If{$\psi^M(x_k)\le \varepsilon$}
\State {\textbf{stop}}

\Else
\State Find $y_{k+1}$ such that:
			$\langle G(y_{k+1}),x_k - y_{k+1}\rangle - \varepsilon> (\psi^M(x_k)-\varepsilon - \delta_k)^+ $
\State $\rho_{k+1}=\rho_k$
\EndIf    
\EndFor
\end{algorithmic}
\end{algorithm}

\begin{theorem} \label{Ideal_theorem}
	Assume  \textbf{(A1)-(A3)} hold true and let  $\{x_k\}_k$ be a sequence generated by Algorithm \ref{alg1}. If the algorithm stops at  iteration  $k$, then $x_{k}$ is an optimal solution to \eqref{OPMYe}. Otherwise, every cluster point of the sequence is an optimal solution to \eqref{OPMYe}.
\end{theorem}
\begin{proof}
        Theorem \ref{Uniform_exact_penalization} guarantees that the while loop (lines $4$-$7$) terminates in a finite number of steps. Suppose the algorithm stops at step $k$, then we have the following chain of inequalities
    \begin{align*}
        f(x_k)&=\min\{f(x) + \rho_k (\psi_k^M(x)-\varepsilon)^+ \ | \ x \in C\} \\
        &\le \min\{f(x) \ | \   \psi^M_k(x) \le \varepsilon,x \in C\} \\
        &\le\min\{f(x) \ | \   \psi^M(x) \le \varepsilon,x \in C\} \\
        &\le f(x_k),
    \end{align*}
    where the last inequality holds since the stopping criterion guarantees that $x_k$ is feasible for \eqref{OPMYe}. As a consequence, it is also optimal for the latter problem. 

    Suppose the algorithm generates an infinite sequence, which means that  $\psi^M(x_k)>{\varepsilon}$ always holds. Consider the half-space
    \begin{align*}
        H(y):=\{(x,\eta)\in \mathbb{R}^{n+1} \ | \ \langle G(y),x - y\rangle \le \eta\}.
    \end{align*}
    Then  $(x_k,\psi^M_{k}(x_k)) \not \in H(y_{k+1})$ holds, otherwise we would get the contradiction
    \begin{align*}
	\varepsilon< \langle G(y_{k+1}),x_k - y_{k+1}\rangle \le \psi^M_k(x_k)\le \varepsilon,
	\end{align*}
    where the first inequality follows from the condition at line $12$ and the last from the stopping condition of the loop at lines $4$-$7$. 
    Since $(x_j,\psi^M_{j}(x_j)) \in H(y_{k+1})$ with $j\ge k+1$ follows from the definition of $\psi^M_{k+1}$, then we have 
    \begin{align*}
	\text{dist}\big ((x_k,\psi^M_{k}(x_k)),(x_j,\psi^M_{j}(x_j))\big)&\ge \text{dist}\big((x_k,\psi^M_{k}(x_k)),H(y_{k+1})\big),
    \end{align*}
    where dist denotes the Euclidean distance. We have the following lower estimates of the above distance 
    \begin{align*}
	\text{dist}\big((x_k,\psi^M_{k}(x_k)),H(y_{k+1})\big)
	&=\dfrac{|\langle G(y_{k+1}),y_{k+1}-x_k\rangle + \psi^M_{k}(x_k) |}{\sqrt{||G(y_{k+1})||^2 + 1}}\\
	&\ge\dfrac{1}{\bar{G}'}{|\langle G(y_{k+1}),y_{k+1} -x_k\rangle + \varepsilon + \psi^M_{k}(x_k)-\varepsilon |}\\
    &=\dfrac{1}{\bar{G}'}\big(|\langle G(y_{k+1}),y_{k+1} -x_k\rangle + \varepsilon| +| \psi^M_{k}(x_k) -\varepsilon|\big)\\
	&\ge \dfrac{1}{\bar{G}'}\big((\psi^M(x_k)-\varepsilon - \delta_k)^+\big),
	\end{align*}
    where $\bar{G}'$ is the maximum of $\sqrt{G(y)^2+1}$ over $C$. The equality in the above chain is provided by the formula of the distance of a point from \textcolor{black}{a} hyperplane, the second equality  from splitting the absolute value of two non positive terms, while  the last inequality  from the choice of $y_{k+1}$ at line $12$. 

    Without loss of generality, we can assume $x_k\rightarrow x_\infty$ for some cluster point $x_\infty \in C$. Taking the limit as $j\rightarrow +\infty$ in the above final estimate  leads to:
    \begin{align*}
	\bar{G}'\text{dist}\big((x_k,\psi^M_{k}(x_k)),(x_\infty,\psi^M_{\infty}(x_\infty)\big)&\ge 
	(\psi^M(x_k) -\varepsilon- \delta_k)^+.
	\end{align*}
     Taking the limit as  $k \rightarrow +\infty$, we get $   (\psi^M(x_\infty) -\varepsilon)^+ \le 0$ which means that $x_\infty$ is feasible for \eqref{OPMYe}. Given any optimal solution  $x^*$ to \eqref{OPMYe}, we have
	$$f(x_{k}) \le f(x^*) \le f(x_\infty),$$
    where the left inequality holds since all the values $f(x_k)$ provide a lower bound to the optimal value of \eqref{OPMYe}, as shown at the beginning of the proof. Taking the limit, we get $f(x^*) = f(x_\infty)$, hence $x_\infty $ is an optimal solution to \eqref{OPMYe}.

\end{proof}

Algorithm \ref{alg1} is a viable option when   the evaluation of the  Minty gap function is possible (for instance, when $G$ is affine) since new cuts are computed using this function. If this is not the case, we can rely on the  Stampacchia gap function  to compute new cuts when the operator $G$ is Lipschitz.

\begin{proposition}\label{Lip}
    Assume \textbf{(A1)} and \textbf{(A3)} hold true and $G$ is  $L$-Lipschitz continuous. Let $D$ be  the diameter of the set $C$ and consider    $0<\varepsilon \le D^2L$ and
	 $x \in C$. If $y \in C$ satisfies
	 $$\langle G(x),x-y \rangle> 2D\sqrt{L \varepsilon },$$
	then there exists $\lambda \in [0,1]$ such that
	$$\langle G(z(\lambda)),x-z(\lambda) \rangle>\varepsilon,$$
	where $z(\lambda)=x + \lambda(y-x)$. In particular, $\bar{\lambda}=\dfrac{\sqrt{\varepsilon}}{D\sqrt{L}}$ guarantees  $z(\bar{\lambda})$ to satisfy the above inequality.  
	 
\end{proposition}
\begin{proof}
The  following chain of inequalities:
 \begin{align*}
 	\langle G(z(\lambda)),x-z(\lambda) \rangle&= \langle G(z(\lambda)) - G(x),x-z(\lambda) \rangle + \langle G(x),x-z(\lambda)  \rangle\\
 	&\ge -\lambda\| G(z(\lambda)) - G(x)\|\cdot\|y -x\| +\lambda \langle G(x),x-y \rangle \\
 	&> -L\lambda^2D^2 +2\lambda D \sqrt{L \varepsilon },
 \end{align*}
 holds thanks to the Cauchy-Schwarz inequality and the assumptions.
 The quantity $-L\lambda^2D^2 - 2\lambda D \sqrt{L \varepsilon }$  achieves the maximum at $ \bar{\lambda}$, 
which yields
\begin{align*}
    \langle G(z(\bar{\lambda})),x-z(\bar{\lambda}) \rangle >- \varepsilon + 2\varepsilon =\varepsilon,
\end{align*}
proving the statement.
\end{proof}

While Proposition \ref{Lip} gives us  a direct way to devise an algorithm that only employs  the Stampacchia gap function   to compute new cuts, it comes at the cost of changing the stopping criterion of Algorithm \ref{alg1} which entails a worse degree of  final inexactness. On the other hand, the new cut  in Algorithm \ref{alg2} is simply given by $z(\bar{\lambda})$ in Proposition \ref{Lip}.

\begin{algorithm}
\caption{}\label{alg2}
\begin{algorithmic}[1]
\State \textbf{input:} $y_0 \in  C$, $\rho_0>0$, $\sigma>1$, $\delta_k>0$ and $ \delta_k\downarrow0$.
\For{$k=0,1,\dots$}
\State $z \in \arg\min\{f(x) + \rho_k (\psi_k^M(x)-\varepsilon)^+\ | \ x \in C\}$
\While{$\psi^{M}_{k} (z)>\varepsilon$}
    \State{$ \rho_{k}=\sigma \rho_k$
    \State$z \in \arg\min\{f(x) + \rho_k (\psi_k^M(x)-\varepsilon)^+\ | \ x \in C\}$}
\EndWhile
\State $x_k=z$
\If{$\psi^S(x_k)\le 2D\sqrt{L \varepsilon } $}
\State {\textbf{stop}}

\Else
\State Find $\bar{y}_k$ such that:
				$
				\langle G(x_k),x_k - \bar{y}_k \rangle - 2D\sqrt{L \varepsilon }> (\psi^S(x_k)-{2D\sqrt{L \varepsilon }} - \delta_k)^+ 
				$\;
\State $y_{k+1}=x_k +\dfrac{\sqrt{\varepsilon}}{D\sqrt{L}}(\bar{y}_k -x_k)$
\State $\rho_{k+1}=\rho_k$
\EndIf    
\EndFor
\end{algorithmic}
\end{algorithm}

We give an alternative  version of the Algorithm \ref{alg2}, which doesn't require the knowledge of the diameter $D$ and the Lipschitz constant $L$.
 As in the previous algorithm, the new cut is  a convex combination of the last  optimal solution of the penalized problem and the  corresponding  maximizer  of the Stampacchia gap function. The difference lies in the fact that the parameter of the convex combination is not fixed but  computed with an exact line search.
This step (line $10$) is the  challenging  part of Algorithm \ref{alg3}, since we have to solve a one-dimensional  maximization problem which is not necessarily concave.

\begin{algorithm}
\caption{}\label{alg3}
\begin{algorithmic}[1]
\State \textbf{input:} $y_0 \in  C$, $\rho_0>0$, $\sigma>1$.
\For{$k=0,1,\dots$}
\State $z \in \arg\min\{f(x) + \rho_k (\psi_k^M(x)-\varepsilon)^+\ | \ x \in C\}$
\While{$\psi^{M}_{k} (z)>\varepsilon$}
    \State{$ \rho_{k}=\sigma\rho_k$
    \State$z \in \arg\min\{f(x) + \rho_k (\psi_k^M(x)-\varepsilon)^+\ | \ x \in C\}$}
\EndWhile
\State $x_k=z$
\State $\bar{y}_k\in \arg\max \{\langle G(x_k),x_k-y\rangle\ | \ y\in C \}$
\State $\bar{\lambda}\in \arg\max\{\lambda \langle G(x_k+\lambda(\bar{y}_k-x_k)),x_k-\bar{y}_k\rangle \ | \ \lambda\in[0,1]\}$
\State  $y_{k+1}=x_k+\bar{\lambda}(\bar{y}_k-x_k) $
\State $\rho_{k+1}=\rho_k$
\If{$\langle G(y_{k+1}),x_k-y_{k+1}\rangle\le \varepsilon$}
\State {\textbf{stop}}
\EndIf    
\EndFor
\end{algorithmic}
\end{algorithm}

As a direct consequence of the proof of  Proposition \ref{Lip}, we get that the following implication holds true for all $x \in C$:
\begin{align*}
    \psi^S(x) > 2D\sqrt{L \varepsilon } \implies \psi^M(x)> \varepsilon.
\end{align*}
Therefore, we gain the following inclusion:
\begin{align*}
\{x \in C \ | \  \psi^M(x)\le \varepsilon\}\subseteq \{x \in C \ | \ \psi^S(x)\le 2D\sqrt{L\varepsilon} \}.
\end{align*} 
While  Algorithm \ref{alg1}  guarantees convergence  to optimal solutions to \eqref{OPMYe}, this is not true for  Algorithms \ref{alg2} and  \ref{alg3}. In fact,   we  are guaranteed to find points which are inside the right hand side of the inclusion, as  stated  in the following theorem. 
\begin{theorem}\label{Theorem_realistic}
    Assume \textbf{(A1)-(A3)} hold true and    $G$ is also $L$-Lipschitz \textcolor{black}{continuous}. Let  $\{x_k\}_{k \in \mathbb{N}}$ be a sequence generated by  Algorithm \ref{alg2} or \ref{alg3}. If the algorithm stops at iteration $k$, then $x_{k}$ satisfies:
    \begin{align*}
    x_{k}\in \arg\min\{f(x)\ | \ x \in C, \ \psi^M_{k}(x) \le \varepsilon\} \bigcap \{x\in C\ | \ \psi^S(x)\le {2D\sqrt{L \varepsilon }}\}.
	\end{align*}
	 Otherwise every   cluster point $x_\infty$ of the  sequence satisfies:
	\begin{align*}
    x_{\infty}\in \arg\min\{f(x)\ | \ x \in C, \ \psi^M_{\infty}(x) \le \varepsilon\} \bigcap \{x\in C\ | \ \psi^S(x)\le {2D\sqrt{L \varepsilon }}\}.
	\end{align*}
\end{theorem}
The  proof follows in the footsteps of  Theorem \ref{Ideal_theorem}, with slight modifications to account for the new stopping criterion, hence it is omitted.

Theorem \ref{Theorem_realistic}  provides bounds for  the optimal values of two different inexact versions of \eqref{OPVI}. In fact,
the final point or any cluster point $\tilde{x}_{\varepsilon}$  of the algorithms satisfies:
\begin{align*}
 \min\{f(x)\ | \ \psi^S(x)\le{2D\sqrt{L \varepsilon }} \}
 \le f(\tilde{x}_{\varepsilon}) \le \min\{f(x)\ | \ \psi^M(x)\le \varepsilon\}.
\end{align*}
In turn, inclusion \eqref{contains} implies the following bounds:
\begin{align*}
 \min\{f(x)\ | \ \psi^S(x)\le{2D\sqrt{L \varepsilon }} \}
 \le f(\tilde{x}_{\varepsilon}) \le \min\{f(x)\ | \ \psi^S(x)\le \varepsilon\}.
\end{align*}
The above  inequalities can be used to show that $\tilde{x}_{\varepsilon}$
behave well when the inexactness goes to zero.

\begin{proposition}
Assume \textbf{(A1)-(A3)} hold true and  consider a sequence $\{\tilde{x}_{\varepsilon_k} \}_{k}$ with $\varepsilon_k \downarrow0$ such that:
   \begin{align*}
   \min\{f(x)\ | \ \psi^S(x)\le{2D\sqrt{L\varepsilon_k }} \}
 \le f(\tilde{x}_{\varepsilon_k}) \le \min\{f(x)\ | \ \psi^S(x)\le \varepsilon_k\}.
\end{align*}
Then every cluster point of the sequence  is an optimal solution to \eqref{OPVI}.
\end{proposition}
\begin{proof}
 Let  $\bar{x}$  be a cluster point of the sequence and $x^*$ an optimal solution to \eqref{OPVI}. The feasible region of \eqref{OPVI} is contained in the feasible region of ($\varepsilon_k$-OPVI), that is 
    \begin{align*}
        \{x \in C\ | \ \psi^S(x)\le 0\} \subseteq\{x \in C\ | \ \psi^S(x)\le \varepsilon_k\},
        \end{align*}
    so that $f(x^*) \ge f(\tilde{x}_{\varepsilon_k})$ holds. Without loss of generality, we can assume $\tilde{x}_{\varepsilon_k} \rightarrow \bar{x}$, so that   $f(x^*)\ge f(\bar{x})$ holds thanks to continuity of $f$. Moreover, we have
        \begin{align*}
           0\ge \lim_{j \rightarrow +\infty} \psi^S(\tilde{x}_{\varepsilon_k})- {2D\sqrt{L{\varepsilon_k} }}= \psi^S(\bar{x}),
        \end{align*}
        which implies that $x^*$ is feasible, hence optimal for \eqref{OPVI}.
\end{proof}

The final value of the Stampacchia gap function measures how much the initial degree inexactness has been worsened  by the approximations  in the algorithms.
In fact, Theorem \ref{Theorem_realistic} guarantees that  the final inexactness is at most $2D\sqrt{L \varepsilon }$, so that: 
\begin{align*}
    \psi^S(\tilde{x}_{\varepsilon})\in[0,2D\sqrt{L \varepsilon }].
\end{align*}
In the next section  the results of some preliminary numerical tests are reported which aim at analyzing the gap between the actual  final inexactness $\tilde{\varepsilon}=\psi^S(\tilde{x}_{\varepsilon}) $ and  the  theoretical bound. 

\section{Numerical results}\label{sec5}

All the numerical tests have been run relying on Algorithm \ref{alg3} since  we only considered problems with non affine monotone operators. First,  we considered  a  problem where a quadratic function is minimized over the set of solution to a variational inequality in order to check how much the final inexactness is better than the    theoretical bound.
Afterwards,  we considered Cournot competition over a network aiming at choosing a equilibrium that   maximizes the social welfare, that is, the sum of the profits of  all the firms. Indeed, \eqref{OPVI} provides a suitable model for  this equilibrium selection problem. The performance of Algorithm 3 on this problem is compared  with the recent algorithms from \cite{SaYu25,KaFa21}.

All the numerical tests have been run on a laptop with an Intel(R) Core(TM) i7-8750H CPU @ 2.20GHz   2.21 GHz with windows $64$ bit,  using \textcolor{black}{Python} $3.8$ and \textcolor{black}{Gurobi Optimizer} $11$. 
We used the DIRECT algorithm \cite{direct} to compute the  maximum  $\bar{\lambda}$ at line $10$.
Also,  we used the   \textcolor{black}{Gurobi Optimizer} to solve the convex minimization problem at line $3$ and $6$.

\textbf{Problem 1:} Consider \eqref{OPVI} where $f$ is quadratic and convex, i.e. $f(x)=\langle (x-u),Q(x-u) \rangle$ with $u \in \mathbb{R}^n$ and $Q \in \mathbb{R}^{n \times n}$  a symmetric positive definite matrix. The operator $G$ is the  sum of one linear and one \textcolor{black}{nonlinear} term, of the form
\begin{align*}
    G(\textbf{x})=\begin{pmatrix}
        M & \textbf{0}_{n-l\times l}\\
        \textbf{0}_{l\times n- l} & \textbf{0}_{l\times l}
    \end{pmatrix} \textbf{x} +  \begin{pmatrix}
        \textbf{b}\\
        \textbf{0}_{l}
    \end{pmatrix} +  
    \begin{pmatrix}
    \alpha_1 e^{\beta_1x^1}\\
    \vdots\\
    \alpha_{n-l} e^{\beta_{n-l}x^{n-l}}\\
    \textbf{0}_{l}
    \end{pmatrix},
\end{align*}
in order to guarantee a lower bound $l$ on the  dimension of the  solution set  of the lower problem.
The matrix $M \in \mathbb{R}^{(n-l)\times (n-l)}$ is non symmetric and positive definite, while 
 $\alpha_i$ and $\beta_i$ are taken positive so that $G$ is  monotone.  
 
 The coefficients $M,\alpha,\beta,b$ were randomly generated and scaled to control  the Lipschitz constant $L$ of $G$.  The starting cutting point  $y_0$ was taken randomly inside the set $C$. Three types of sets $C$ were considered: the unitary cube,  sphere 
 and  simplex.
 The parameter $\sigma$, which regulates the rate at which the penalty parameter grows, had a value of  $1.2$, while the starting penalty parameter $\rho_0$  was set at $1$. 
 
 Each table reports the average results on $100$ random instances. To test the speed, we measured the computational time in seconds, the number of times $\rho$ increased and the number of cuts that have been added. In order  to test the accuracy, we measured the value $\tilde{\varepsilon}$ of the Stampacchia gap function at the final point of the algorithm and  compared it  with the theoretical bound $E=2D\sqrt{L\varepsilon}$.

\begin{table}[htbp]
    \centering
    \resizebox{10cm}{!}{%
    \begin{tabular}{cccccccc}
    \toprule
        $C$ &$\|b\|$&time (sec)& $\rho$ incr.&  cuts  &  $\tilde{\varepsilon}$ & $E$& $\tilde{\varepsilon}/E$\\
    \midrule[0.8pt]
    \multirow{5.4}{*}{Cube} &0 &0.4356&19.96&8.33&0.0102&6.3246&0.0016\\
    \cmidrule{2-8}
    &$\sqrt{n}$ &1.1866&23.25&24.51&0.063&6.3246&0.01\\
    \cmidrule{2-8}
    &$\sqrt{2n}$ &5.6749&25.27&105.64&0.4592&6.3246&0.0726\\
    \cmidrule{2-8}
     &$\sqrt{3n}$ &7.4718&28.36&150.31&0.6731&6.3246&0.1064\\
    \midrule
    
    \multirow{5.4}{*}{Sphere} &$0$  &0.3254&32.62&11.7&0.0235&1.7889&0.0132\\
    
    \cmidrule{2-8}
    &$\sqrt{n}$ &0.315&32.52&10.75&0.0222&1.7889&0.0124\\
    \cmidrule{2-8}
     &$\sqrt{2n}$ &0.3897&32.47&9.68&0.0203&1.7889&0.0114\\
    \cmidrule{2-8}
     & $\sqrt{3n}$ &0.3173&32.26&8.13&0.0181&1.7889&0.0101\\
    \midrule
    \multirow{5.4}{*}{Simplex}  &$0$ &0.0752&11.73&2.97&0.0093&1.2649&0.0073\\
    \cmidrule{2-8}
    &$\sqrt{n}$ &0.1996&12.05&14.16&0.0602&1.2649&0.0476\\

    \cmidrule{2-8}
    &$\sqrt{2n}$ &0.731&23.89&44.85&0.3218&1.2649&0.2544\\

    \cmidrule{2-8}
    &$\sqrt{3n}$ &0.6899&27.57&48.68&0.3548&1.2649&0.2805\\

    \bottomrule
    \end{tabular}}
    \caption{Results for different values of $ \|b\|$, $n=50$, $l=10$, $L=20$, $\varepsilon=0.01$.}
    \label{tab1}
\end{table}
Table \ref{tab1} focuses on the sensitivity with respect  to the coefficient $b$ which have been  uniformly drawn  from  a sphere of a given radius.  The algorithm performed worse as $ \|b\|$ grew  on the cube and the simplex, while the opposite  happened  on the sphere.

\begin{table}[htbp]
    \centering
    \resizebox{10cm}{!}{%
    \begin{tabular}{cccccccc}
    \toprule
        $C$ &$L$&time (sec)& $\rho$ incr.&  cuts  &  $\tilde{\varepsilon}$ & $E$& $\tilde{\varepsilon}/E$\\
    \midrule[0.8pt]
    \multirow{4.2}{*}{Cube} & $20$ &4.9364&10.41&10.75&0.5407&6.3246&0.0855\\
    \cmidrule{2-8}
    &$50$ &20.5215&7.12&64.36&0.8751&10.0&0.0875\\
    \cmidrule{2-8}
    &$100$ &88.2439&4.29&608.14&1.2529&14.1421&0.0886\\
    \midrule
    
    \multirow{4.2}{*}{Sphere}&$20$  &0.2113&15.92&8.89&0.0193&1.7889&0.0108\\
    
    \cmidrule{2-8}
    &$50$ &0.7579&14.12&48.18&0.0732&2.8284&0.0259\\
    \cmidrule{2-8}
    &$100$ &8.2735&12.93&207.48&0.4819&4.0&0.1205\\
    \midrule
    \multirow{4.2}{*}{Simplex}&$20$ &0.4068&11.0&39.18&0.3243&1.2649&0.2564\\
    \cmidrule{2-8}
    &$50$ &0.9597&9.74&64.63&0.5163&2.0&0.2581\\
    \cmidrule{2-8}
    &$100$ &1.3257&7.27&87.63&0.7464&2.8284&0.2639\\

    \bottomrule
    \end{tabular}}
    \caption{Results for different values of  $L$, $n=50$, $l=10$, $ \|b\|=5\sqrt{6}$, $\varepsilon=0.01$.}
    \label{tab2}
\end{table}

Table \ref{tab2} focuses on the sensitivity with respect to the Lipschitz constant $L$. Though the final inexactness worsen as $L$ increases,  the ratio $\tilde{\varepsilon}/E$ stayed  quite   stable on the cube and the simplex. The number of times the penalty parameter increases went down as $L$ increased, while the number of cuts grew. Notice that the growth of the number of cuts brought in an   increase in computational time.

\begin{table}[htbp!]
    \centering
    \resizebox{10cm}{!}{%
    \begin{tabular}{cccccccc}
    \toprule
        $C$ &$\varepsilon$&time (sec)& $\rho$ incr.&  cuts  &  $\tilde{\varepsilon}$ & $E$& $\tilde{\varepsilon}/E$\\
    \midrule[0.8pt]
    \multirow{5.4}{*}{Cube} &$1$ &1.0865&15.83&27.38&7.6783&63.2456&0.1214\\
    \cmidrule{2-8}
    &$0.1$ &2.3805&20.16&57.28&1.6634&20.0&0.0832\\
    \cmidrule{2-8}
    &$0.01$ &5.6103&25.98&99.66&0.4432&6.3246&0.0701\\
    \cmidrule{2-8}
     &$0.001$ &7.5929&31.4&124.07&0.1217&2.0&0.0609\\
    \midrule
    \multirow{5.4}{*}{Sphere} &$1$  &0.1943&19.39&8.22&2.0095&17.8885&0.1123\\
    \cmidrule{2-8}
    &$0.1$  &0.256&26.05&9.36&0.2052&5.6569&0.0363\\
    \cmidrule{2-8}
    &$0.01$  &0.4113&32.55&9.05&0.0198&1.7889&0.0111\\
    \cmidrule{2-8}
    &$0.001$  &0.4174&38.76&8.9&0.0019&0.5657&0.0034\\
    \midrule
    \multirow{5.4}{*}{Simplex} &$1$ &0.0924&3.59&7.85&3.8881&12.6491&0.3074\\
    \cmidrule{2-8}
    &$0.1$ &0.3064&11.48&29.87&1.141&4.0&0.2853\\
    \cmidrule{2-8}
    &$0.01$ &0.6528&25.56&36.28&0.3256&1.2649&0.2574\\
    \cmidrule{2-8}
    &$0.001$
    &0.9102&32.95&53.75&0.103&0.4&0.2575\\
    \bottomrule
    \end{tabular}}
    \caption{Results for different values $\varepsilon$, $n=50$, $l=10$, $ \|b\|=10$, $L=20$.}
    \label{tab3}
\end{table}

Table \ref{tab3} shows the behavior of the algorithm as $\varepsilon$ decreases towards $0$. As it is expected, the actual inexactness  $\tilde{\varepsilon}$ decreased with $\varepsilon$. Moreover, the ratio $\tilde{\varepsilon}/E$ decreased meaningfully on the sphere and the cube, while it remained stable on the simplex.
\begin{table}[h]
    \centering
    \resizebox{10cm}{!}{%
    \begin{tabular}{cccccccc}
    \toprule
        $C$ &$n$&time (sec)& $\rho$ incr.&  cuts  &  $\tilde{\varepsilon}$ & $E$& $\tilde{\varepsilon}/E$\\
    \midrule[0.8pt]
    \multirow{5.4}{*}{Cube} &$25$ &0.8157&23.36&29.55&0.2408&4.4721&0.0538\\
    \cmidrule{2-8}
     &$50$ &3.7345&25.47&84.29&0.435&6.3246&0.0688\\
    \cmidrule{2-8}
     &$75$ &11.0466&26.11&156.22&0.5451&7.746&0.0704\\
    \cmidrule{2-8}
    &$100$ &26.0887&27.05&209.88&0.6628&8.9443&0.0741\\
    \midrule
    \multirow{5.4}{*}{Sphere}&$25$ &0.2701&30.99&21.59&0.0344&1.7889&0.0192\\
    
   \cmidrule{2-8}
    &$50$ 
    &0.275&32.47&9.15&0.02&1.7889&0.0112\\
   \cmidrule{2-8}
    &$75$ &0.4611&33.23&7.09&0.0167&1.7889&0.0093\\
    \cmidrule{2-8}
    &$100$  &0.8179&33.69&6.77&0.0161&1.7889&0.009\\
    \midrule
    \multirow{5.4}{*}{Simplex}&$25$ &0.1707&23.27&16.96&0.2319&1.2649&0.1834\\
    \cmidrule{2-8}
    &$50$ &0.4307&26.66&35.66&0.3205&1.2649&0.2534\\
    \cmidrule{2-8}
    &$75$ &1.0409&27.46&57.23&0.3463&1.2649&0.2738\\
    \cmidrule{2-8}
    &$100$
    &2.2926&28.04&75.71&0.3655&1.2649&0.289\\
    \bottomrule
    \end{tabular}}
    \caption{ Results for different values $n$, $l=\dfrac{n}{5}$, $ \|b\|=\sqrt{2n}$, $L=20$.}
    \label{tab4}
\end{table}

Table \ref{tab4} shows the behavior of the algorithm as the dimension $n$ changes while keeping the   lower bound $l=n/5$. Notice that the final inexactness $\tilde{\varepsilon}$ increased on the cube while it did not change meaningfully on the sphere and the simplex. This behavior is expected, since each cut ``eliminates" some vertices and the number of vertices increases exponentially for the cube, while it is just $n+1$  for the simplex. Finally, the ratio $\tilde{\varepsilon}/E$ seemed quite stable while depending on the type of problem.

\begin{table}[h]
    \centering
    \resizebox{10cm}{!}{%
    \begin{tabular}{cccccccc}
    \toprule
        $C$ &$\rho_0$&time (sec)& $\rho$ incr.&  cuts  &  $\tilde{\varepsilon}$ & $E$& $\tilde{\varepsilon}/E$\\
    \midrule[0.8pt]
    \multirow{5.4}{*}{Cube}
     
     &$100$ &4.0306&1.83&95.0&0.4335&6.3246&0.0685\\
    
    \cmidrule{2-8}
     &$500$ &3.9787&0.0&92.63&0.4275&6.3246&0.0676\\
    \cmidrule{2-8}
    &$1000$ &4.2294&0.0&97.0&0.4336&6.3246&0.0686\\
    \cmidrule{2-8}
    &$2000$ &5.5133&0.0&94.39&0.4379&6.3246&0.0692\\
    \midrule
    \multirow{5.4}{*}{Sphere}

     &$100$ &0.1604&7.22&9.83&0.0209&1.7889&0.0117\\

   \cmidrule{2-8}
    &$500$ &0.1099&0.0&9.79&0.0209&1.7889&0.0117\\
    \cmidrule{2-8}
    &$1000$  &0.113&0.0&9.78&0.0209&1.7889&0.0117\\
    \cmidrule{2-8}
    &$2000$  &0.1415&0.0&9.81&0.0209&1.7889&0.0117\\
    \midrule
    \multirow{5.4}{*}{Simplex}
     &$100$ &0.4055&3.82&36.72&0.3217&1.2649&0.2543\\
    \cmidrule{2-8}
    &$500$ &0.3862&0.01&35.61&0.3185&1.2649&0.2518\\
    \cmidrule{2-8}
    &$1000$
    &0.4198&0.0&36.37&0.3186&1.2649&0.2518\\

     \cmidrule{2-8}
    &$2000$
    &0.6915&0.0&35.85&0.3209&1.2649&0.2537\\
    \bottomrule
    \end{tabular}}
    \caption{ Results for different values $\rho_0$, $n=50,l=10$, $ \|b\|=10$, $L=20$, $\varepsilon=0.01$.}
    \label{tab5}
\end{table}

\textcolor{black}{Table \ref{tab5} reports the behavior of the algorithm  for different choices of the initial penalty parameter $\rho_0$. For initial values  $500$, $1000$ and $2000$, the algorithm  actually becomes akin to a fixed penalty approach, since the penalty parameter almost never increases. On the contrary, small values for $\rho_0$  require to solve  the minimization
problem of Algorithm \ref{alg3}, line 6, many times, incurring in a higher computational burden as shown in Table \ref{tab3}. Moreover, we observe that $\rho_0 = 500$ yields the best runtime and accuracy, suggesting that selecting  very large penalty parameters is unlikely to bring any advantage.} 

All the tests show that the final inexactness is meaningfully lower than the theoretical one, and it amounts to at most  $10\%$ of the theoretical value on the cube and sphere, while at most $30\%$ on the simplex.

{\textbf{Problem 2:} Consider Cournot competition over a network where $N$ firms can produce and sell a single good in different  $J$ locations (see, for instance, \cite{KaFa21,KaSh12}).  The   price at which they sell depends only on the total quantity sold at that location, and supplies can be transported  across the network between locations with negligible transportation costs.
Each firm $ i $ chooses the production level $ y_{ij} $ and the amount of sales $ s_{ij} $ at location $ j $. Therefore, the firm's decision vector is  $ x^{(i)} := (y_i, s_i) \in \mathbb{R}^{2J} $, where  $ y_i := (y_{i1}, \ldots, y_{iJ}) $ and $ s_i := (s_{i1}, \ldots, s_{iJ}) $. 
The firm seeks to maximize its own profit given by
$$
g_i(x^{(i)}, x^{(-i)}) :=   \sum_{j=1}^J s_{ij} \, p_j\left( \sum_{i=1}^N s_{ij} \right) - \sum_{j=1}^J c_{ij}y_{ij},
$$  
where $ c_{ij} $ is the unitary production cost and $
p_j(t) := a_j - b_jt^\sigma
$ is the inverse demand function  that provides the maximum unitary price at which a total quantity $t$ can be sold at location $j$. 
 Furthermore, the firm  has a production capacity of  $ B_{ij} $ at location $j$, and its overall production must match total sales at all locations. Hence, the set of  feasible decisions of the firm  is  
$$
C_i := \left\{ (y_i, s_i) \in \mathbb{R}_+^{2J} \ \middle|\ \sum_{j=1}^J s_{ij} = \sum_{j=1}^J y_{ij}, \ y_{ij} \le B_{ij} \right\}.
$$  
Nash-Cournot equilibria of the problem coincide with the solutions to \eqref{VI} with the operator 
$$
G(x) = \begin{pmatrix}
    \vdots\\
    \nabla_{x^{(i)}}g_i(x^{(i)}, x^{(-i)})\\
    \vdots
\end{pmatrix}
$$  
over the set $C = \prod_{i=1}^N C_i$ (see, for instance \cite{FacchineiPang}). 
Since there could be multiple   equilibria, the maximization of the social welfare
$$f(x) := \sum_{i=1}^N g_i(x^{(i)}, x^{(-i)})$$ 
over the set of equilibria is performed to select one.}

{
In the tests, we considered a network where four firms compete over three locations, i.e. \textcolor{black}{$(N,J)=(4,3)$}. {Unitary costs have been set by} taking values between $0.1$ and $1$ with uniform probability. We set $\sigma = 1.05$, so that $G$ is monotone (see Section 4 in \cite{KaSh12}) and $f$ is convex. We chose $a_j = 1$, $b_j = 0.01$ and $B_{ij} = 5$ for all $i$ and $j$.}

{We ran Algorithm \ref{alg3} and two recent algorithms from the literature: Algorithm~3.1 in \cite{KaFa21} (a-IRG), that employs a Tikhonov penalty approach to solve problem  \eqref{OPVI}; Algorithm~1 in \cite{SaYu25} (IR-EG), that is an extragradient method  for nested variational inequalities. The latter can be exploited as well since the upper level of the equilibrium selection problem  can be recast as a variational inequality leveraging its first order optimality condition, due to the differentiability of  $f$.} 

{In Algorithm \ref{alg3} we set $\varepsilon = 10^{-6}$. All the three algorithms were stopped when the iterate $x_{k+1}$ satisfied both conditions $ | f(x_{k+1}) - f(x_k) | \le 10^{-3}$ and  $\psi^S(x_{k+1}) \le 10^{-2}$. The former provides a standard measure of the lack of meaningful progress in the optimization process, while the latter provides an estimate of the infeasibility of the iterate.}
{Also, a maximum runtime of \textcolor{black}{$300$} seconds was set, since using a higher {threshold} did not yield significantly different results. {Indeed, IR-EG stopped before reaching it in almost every instance, while} a-IRG failed to satisfy the stopping conditions even when the maximum runtime was  extended to \textcolor{black}{600} seconds, \textcolor{black}{ a behavior that  is consistent with the theoretical bound given in \cite[Corollary 3.5]{KaFa21}.}}

Table \ref{tab6} reports the average and maximum runtime, 
the average final inexactness $\tilde{\varepsilon}$, and the average social welfare over $100$ instances. 
All the algorithms found {solutions} 
with similar values 
of the social welfare. 
Algorithm~\ref{alg3} and IR-EG consistently reached  the desired level of inexactness, with the former being  faster both in terms of average and maximum runtime. 
{On the contrary}, a-IRG never reached the required cut-off in the \textcolor{black}{$300$} seconds time frame, yielding a higher final  inexactness.

    

    

\begin{table}[htbp!]
    \centering
    \begin{tabular}{cccccccc}
    \toprule
    Algorithm & Mean time (sec.)  & Max time (sec.)  & $\tilde{\varepsilon}$ &  Social welfare \\
    \midrule[0.8 pt]
    Alg \ref{alg3}&6.15  &  42.94&   0.008  & 17.096  \\
    \midrule[0.8 pt]
    
    a-IRG  & 300.00   & 300.00  &  0.05  & 17.112  \\
    \midrule[0.8 pt]
    IR-EG   & 62.70  & 106.05    & 0.010 & 17.098\\
    \bottomrule 
    \end{tabular}
    \caption{ Performance  on equilibrium selection in Cournot competition~\textcolor{black}{$((N,J)=(4,3))$}}
    \label{tab6} 
\end{table}
\textcolor{black}{ We also compared  Algorithm \ref{alg3} and IR-EG on  instances of different sizes of the problem.  Table \ref{tab7} reports the performance of the two algorithms for networks with different  number of firms $(N)$ and  locations $(J)$ from Table \ref{tab6}.}
\textcolor{black}{Moreover, Figure \ref{fig3} provides the box plot for the runtime of the two algorithms corresponding to the results  of Tables \ref{tab6} and \ref{tab7}, where  the orange line indicates the median while  the green dotted line indicates the mean.}
\begin{table}[htbp!]
\centering
\begin{tabular}{c c c c c c}
\toprule
&&&\multicolumn{3}{c}{Time (sec.)}\\
\cmidrule{4-6} Algorithm & $(N,J)$ &Solved \hspace{0.3cm}&\hspace{0.3cm} Mean \hspace{0.3cm} & \hspace{0.3cm} Median \hspace{0.3cm}  &\hspace{0.3cm} Max \hspace{0.3cm}  \\
\midrule
\multirow{4}{*}{Alg \ref{alg3}} & (5,4) & 100 & 16.00 & 10.08& 189.38 \\
\cmidrule{2-6}
& (6,5) & 100 & 32.29 &20.29& 217.34 \\
\cmidrule{2-6}
 & (7,5) & 98  & 60.24 &36.83& 300.00 \\
\midrule
\multirow{4}{*}{IR-EG}& (5,4) & 100 & 122.86 &120.82& 191.59 \\
\cmidrule{2-6}
 & (6,5) & 95  & 223.63 &225.58& 300.00 \\
 \cmidrule{2-6}
& (7,5) & 70  & 255.61 &262.90& 300.00 \\
\bottomrule

\end{tabular}
\caption{Performance comparison between Alg 3 and IR-EG.}
\label{tab7}
\end{table}

\textcolor{black}{
 Algorithm~\ref{alg3} and IR-EG consistently reached  the desired level of inexactness for the  cases $(5,4)$ and $(6,5)$, while IR-EG fails to solve a significant number of instances in the case $(7,5)$ within the time threshold. Although Algorithm \ref{alg3} is  considerably  faster than IR-EG in general, it exhibits
a meaningful difference between median and mean time, indicating a higher variance in the performance of the algorithm, while IR-EG does not.
 }

\begin{figure}[h]
    \centering
    \includegraphics[width=0.9\linewidth]{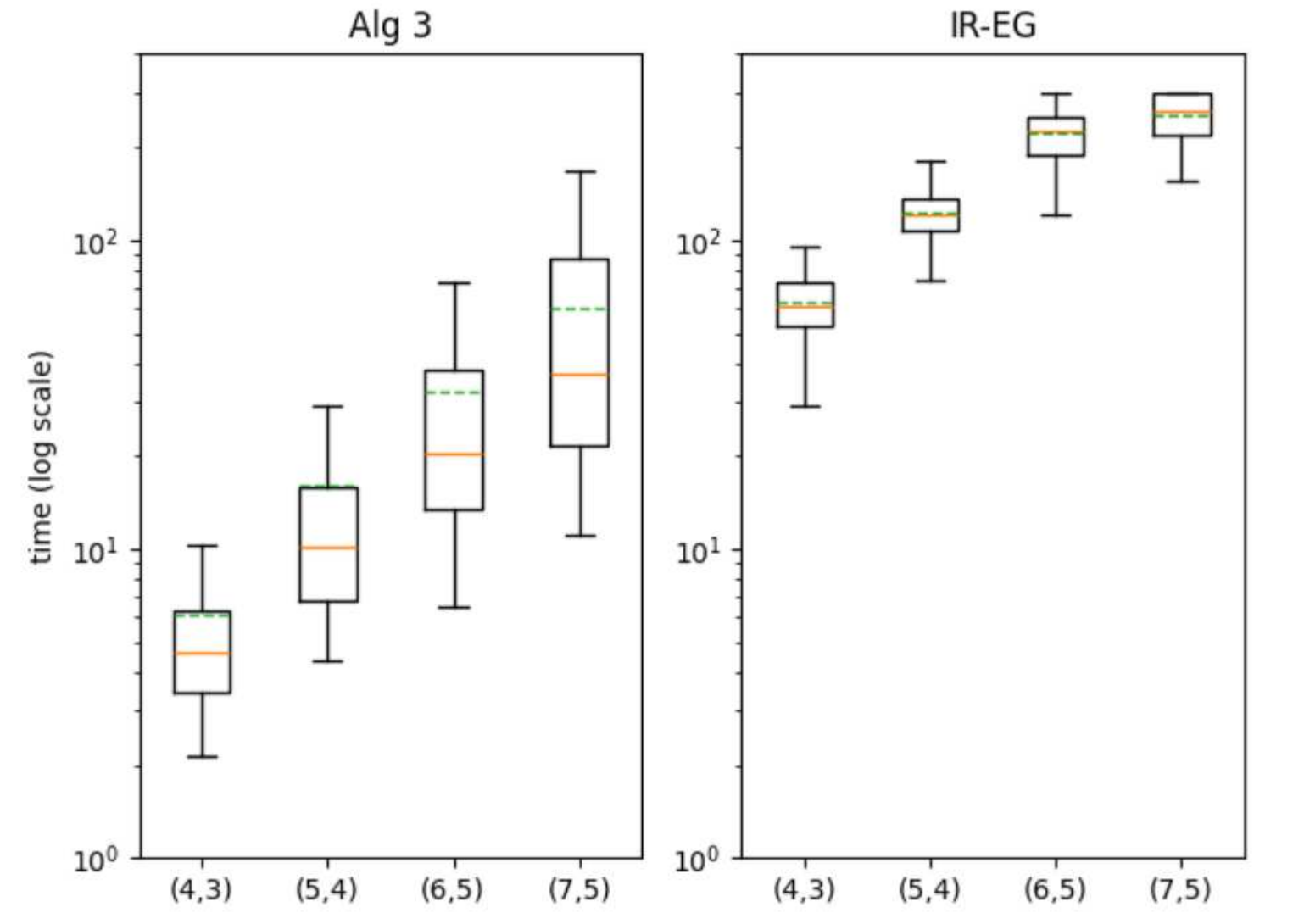}
    \caption{Box plot of the runtime of the two algorithms for different $(N,J)$.}
    \label{fig3}
\end{figure}

\section*{Conclusions}

We combined cutting planes techniques and exact penalization to study  simple  bilevel variational problems, wherein the upper level is an optimization problem and the lower level is a variational inequality. A degree of inexactness in the lower level was exploited to achieve exact penalization. Furthermore,  polyhedral approximations  were considered for the computations in the lower level while preserving exact penalization uniformly. Finally, suitable algorithms have been devised relying on these tools.

The paper aims at providing a basic setting for exploiting inexactness in  bilevel variational problems.
Further work can be carried out in some directions.

While the algorithms in the paper exploit a  degree of inexactness that is fixed at the beginning, adaptive rules to decrease it during the algorithm could be developed.
Since the exact penalty parameter depends on the inexactness, this could require
the integration with some kind of Tikhonov techniques.

The approach requires the solution of a penalized problem at each iteration, which becomes more and more computationally expensive as the number of cuts grows. While  an algorithm that  performs only some steps of a solution method  might lead to better performances, it brings in  additional difficulties in the identification of new cuts. 

\textcolor{black}{When the operator $G$ is not monotone, the approach of this paper would need to be extended by including suitable tools from non-convex and non-smooth optimization. In fact,  \eqref{OPVI} and \eqref{OPMVI} are no longer equivalent and the optimal value of the latter is only an upper bound for the former.  Moreover, the  reformulation of the   lower level with the Stampacchia gap function leads to  an optimization problem that  is non-convex and likely non-smooth. As a consequence, the inexactness alone is not enough to guarantee exact penalization since the existence of a Slater point does not yield (MFCQ) on the whole feasible region. Anyway, the extended version \eqref{eMFCQ}, that holds beyond the feasible region, is enough to guarantee exact penalization  in terms of both local and global minima  (see \cite{Ye12}). }

Finally, the techniques of this paper  could be possibly adapted for other types of bilevel variational problems, where the lower and/or  upper level are, for instance, Nash equilibrium problems with non-smooth data, generalized Nash equilibrium problems, or, more in general, quasi variational or quasi Ky Fan inequalities.\\

\textbf{Acknowledgements}

The authors are members of the Gruppo Nazionale per l’Analisi
Matematica, la Probabilit\`a e le loro Applicazioni (GNAMPA - National Group for Mathematical
Analysis, Probability and their Applications) of the Istituto Nazionale di Alta
Matematica (INdAM - National Institute of Higher Mathematics).\\

\textbf{Data availability}

No data sets were generated during the current study. The used python codes are available on GitHub at \url{https://github.com/RIKTOMA/OPVI} with access upon request.


\bibliographystyle{elsarticle-num} 
\bibliography{bibliography_revised}

@article{Ca05,
	author = {Cabot, Alexandre},
	title = {Proximal Point Algorithm Controlled by a Slowly Vanishing Term: Applications to Hierarchical Minimization},
	journal = {SIAM J. Optim.},
	volume = {15},
    number={2},
	pages = {555-572},
	year = {2005},
        doi={10.1137/S105262340343467X}

}

@article{So07,
	title={An explicit descent method for bilevel convex optimization},
	author={Solodov, Mikhail},
	journal={J. Convex Anal.},
	volume={14},
    number={2},
	pages={227--237},
	year={2007},
	publisher={HELDERMANN VERLAG LANGER GRABEN 17, 32657 LEMGO, GERMANY}
}

@INPROCEEDINGS{AnYo19,
	title={An iterative regularized incremental projected subgradient method for a class of bilevel optimization problems},
	author={Amini, Mostafa and Yousefian, Farzad},
	booktitle={2019 American Control Conference (ACC)},
	pages={4069--4074},
	year={2019},
    doi={10.23919/ACC.2019.8814637}
}

@article{BeSa14,
	title={A first order method for finding minimal norm-like solutions of convex optimization problems},
	author={Beck, Amir and Sabach, Shoham},
	journal={Math. Program.},
	volume={147},
    number={1},
	pages={25--46},
	year={2014},
	publisher={Springer-Verlag Berlin Heidelberg and Mathematical Optimization Society 2013},
        doi={ 10.1007/s10107-013-0708-2}
}

@article{SaSh17,
	author = {Sabach, Shoham and Shtern, Shimrit},
	title = {A First Order Method for Solving Convex Bilevel Optimization Problems},
	journal = {SIAM J. Optim.},
	volume = {27},
    number={2},
	pages = {640-660},
	year = {2017},
        doi={10.1137/16M105592X}
}

@article{ShVuZe21,
	author = {Yekini Shehu and Phan Tu Vuong and Alain Zemkoho},
	title = {An inertial extrapolation method for convex simple bilevel optimization},
	journal = {Optim. Methods Softw.},
	volume = {36},
	pages = {1-19},
    number={1},
	year  = {2021},
	publisher = {Taylor & Francis},
	doi={10.1080/10556788.2019.1619729},
	
	}

@article{TAN2022106160,
author = {Bing Tan and Sun Young Cho},
title = {Two adaptive modified subgradient extragradient methods for bilevel pseudomonotone variational inequalities with applications},
journal = {Commun. Nonlinear Sci. Numer. Simul.},
volume = {107},
pages = {106--160},
year = {2022},
issn = {1007-5704},
doi={10.1016/j.cnsns.2021.106160}
}

@article{KaFa21,
author = {Kaushik, Harshal D. and Yousefian, Farzad},
title = {A Method with Convergence Rates for Optimization Problems with Variational Inequality Constraints},
journal = {SIAM J. Optim.},
volume = {31},
number={3},
pages = {2171-2198},
year = {2021},
doi={10.1137/20M1357378}
}

@article{HiAn22,
author = {Pham Van Huy, Le Huynh My Van, Nguyen Duc Hien and Tran Viet Anh},
title = {Modified Tseng's extragradient methods with self-adaptive step size for solving bilevel split variational inequality problems},
journal = {Optimization},
volume = {71},
number={6},
pages = {1721-1748},
year = {2022},
publisher = {Taylor & Francis},
doi={10.1080/02331934.2020.1834557}
}

@article{AnAn19,
	author = {Pham Ngoc Anh and Le Thi Hoai An},
	title = {New subgradient extragradient methods for solving monotone bilevel equilibrium problems},
	journal = {Optimization},
	volume = {68},
    number={11},
	pages = {2099-2124},
	year = {2019},
	publisher = {Taylor & Francis},
        doi={10.1080/02331934.2019.1656204}
	
}

@article{Lampariello2020,
author={Lampariello, Lorenzo
and Neumann, Christoph
and Ricci, Jacopo M.
and Sagratella, Simone
and Stein, Oliver},
title={An explicit Tikhonov algorithm for nested variational inequalities},
journal={Comput. Optim. Appl.},
year={2020},
number={2},
month={Nov},
day={01},
volume={77},
pages={335-350},
issn={1573-2894},
doi={10.1007/s10589-020-00210-1}
}

@article{Lampariello2022,
author={Lampariello, Lorenzo
and Priori, Gianluca
and Sagratella, Simone},
title={On the solution of monotone nested variational inequalities},
journal={Math. Method Oper. Res.},
year={2022},
volume={96},
number={3},
pages={421-446},
issn={1432-5217},
doi={10.1007/s00186-022-00799-5}
}

@article{Bigi2023,
author={Bigi, Giancarlo
and Lampariello, Lorenzo
and Sagratella, Simone
and Sasso, Valerio Giuseppe},
title={Approximate variational inequalities and equilibria},
journal={Comput. Manag. Sci.},
year={2023},
volume={20},
number={1},
pages={20--43},
issn={1619-6988},
doi={10.1007/s10287-023-00476-w}
}

@article{Bigi2022,
author = {Giancarlo Bigi and Lorenzo Lampariello and Simone Sagratella},
title = {Combining approximation and exact penalty in hierarchical programming},
journal = {Optimization},
volume = {71},
number={8},
pages = {2403-2419},
year = {2022},
doi={10.1080/02331934.2021.1939336}
}

@book{FacchineiPang,
    title =     {Finite Dimensional Variational Inequalities and Complementarity Problems volume 1},
    author =    {Francisco Facchinei and  Jong-Shi Pang},
    publisher = {Springer},
    isbn =      {9780387955803; 0387955801},
    year =      {2003},
    edition =   {1},
    address={New York}
}

@article{Nguyen84,
 ISSN = {00411655, 15265447},
 author = {Sang Nguyen and Clermont Dupuis},
 journal = {Transp. Sci.},
 pages = {185--202},
 publisher = {INFORMS},
 title = {An Efficient Method for Computing Traffic Equilibria in Networks with Asymmetric Transportation Costs},
 volume = {18},
 number={2},
 year = {1984},
doi={10.1287/trsc.18.2.185}
}

@article{Zangwill1971,
author = {Eaves, B. Curtis and Zangwill, W. I.},
title = {Generalized Cutting Plane Algorithms},
journal = {SIAM J. Control. Optim.},
volume = {9},
number={4},
pages = {529-542},
year = {1971},
doi={10.1137/0309037}
}

@book{rockafellar2009variational,
  title={Variational Analysis},
  author={R. Tyrrell Rockafellar and Roger J. B. Wets},
  isbn={9783642024313},
  series={Grundlehren der mathematischen Wissenschaften},
  year={2009},
  publisher={Springer},
  address={Berlin, Heidelberg}
}

@article{XU2004279,
title = {Viscosity approximation methods for nonexpansive mappings},
journal = {J. Math. Anal. Appl.},
volume = {298},
number={1},
pages = {279-291},
year = {2004},
issn = {0022-247X},
author = {Hong-Kun Xu},
doi={10.1016/j.jmaa.2004.04.059}

}

@article{Korpelevich1976TheEM,
  title={The extragradient method for finding saddle points and other problems},
  author={Korpelevich, Galina M},
  journal={Matecon},
  volume={12},
  pages={747--756},
  year={1976}
}

@book{hausdorff1991set,
  title={Set Theory},
  author={Hausdorff, F.},
  isbn={9780828401197},
  lccn={57008493},
  series={AMS/Chelsea Series},
  year={1991},
  publisher={Chelsea Publishing Company}
}

@article{bigi2010successive,
  title={A successive linear programming algorithm for nonsmooth monotone variational inequalities},
  author={Bigi, Giancarlo and Panicucci, Barbara},
  journal={Optim. Methods Softw.},
  volume={25},
  number={1},
  pages={29--35},
  year={2010},
  publisher={Taylor \& Francis},
    doi={10.1080/10556780903157885}
}

@article{FePa97,
author = {Ferris, M. C. and Pang, J. S.},
title = {Engineering and Economic Applications of Complementarity Problems},
journal = {SIAM Rev.},
volume = {39},
number={4},
pages = {669-713},
year = {1997},
doi={10.1137/S0036144595285963}
}

@Book{Rudin1976,
author={Rudin, Walter},
title={Principles of mathematical analysis},
series={International series in pure and applied mathematics},
year={1976},
edition={Third edition. International edition},
publisher={McGraw-Hill},
address={New York},
isbn={007054235X; 9780070542358; 0070856133; 9780070856134; 9781259064784; 1259064786},
language={English}
}

@book{DeZe20,
    author ={ Dempe,Stephan and Zemkoho, Alain},
    title ={ Bilevel Optimization},
    publisher = {Springer},
    year = {2020},
    edition={1},
    series={Springer Optimization and Its Applications},
    address={Cham}
}

@article{KaSaYo23,
  author={Kaushik, Harshal D. and Samadi, Sepideh and Yousefian, Farzad},
  journal={IEEE Trans. Autom. Control}, 
  title={An Incremental Gradient Method for Optimization Problems With Variational Inequality Constraints}, 
  year={2023},
  volume={68},
  number={12},
  pages={7879-7886},
doi={10.1109/TAC.2023.3251851}
}

@article{Giang-Tran2024,
author={Giang-Tran, Khanh-Hung
and Ho-Nguyen, Nam
and Lee, Dabeen},
title={A projection-free method for solving convex bilevel optimization problems},
journal={Math. Program.},
number={1},
year={2024},
issn={1436-4646},
doi={10.1007/s10107-024-02157-1}
}

@InProceedings{JiAbMoHa23,
title = "A Conditional Gradient-based Method for Simple Bilevel Optimization with Convex Lower-level Problem",
author = "Ruichen Jiang and Nazanin Abolfazli and Aryan Mokhtari and Hamedani, {Erfan Yazdandoost}",
year = "2023",
language = "English (US)",
volume = "206",
pages = "10305--10323",
booktitle = 	 {Proceedings of The 26th International Conference on Artificial Intelligence and Statistics},
series = 	 {Proceedings of Machine Learning Research},

issn = "2640-3498",
publisher = "PMLR",
}

@article{BeSte24,
author="Mare Beck and Oliver Stein",
title="Semi-Infinite Models for Equilibrium Selection",
volume={9}, 
journal={Minimax Theory  Appl.}, year={2024}, month={Jan.},
pages={1–18},

}

@article{Ferris1991,
author={Ferris, M. C.
and Mangasarian, O. L.},
title={Finite perturbation of convex programs},
journal={Appl. Math. Optim.},
year={1991},
month={Jan},
number={1},
volume={23},
pages={263-273},
issn={1432-0606},
doi={10.1007/BF01442401}
}

@article{SaYu25,
author = {Samadi, Sepideh and Yousefian, Farzad},
title = {Improved Guarantees for Optimal Nash Equilibrium Seeking and Bilevel Variational Inequalities},
journal = {SIAM J. Optim.},
volume = {35},
number={1},
pages = {369-399},
year = {2025},
doi = {10.1137/23M1589402},

}

@article{KaSh12,
author = {Kannan, Aswin and Shanbhag, Uday V.},
title = {Distributed Computation of Equilibria in Monotone Nash Games via Iterative Regularization Techniques},
journal = {SIAM J. Optim.},
volume = {22},
number = {4},
pages = {1177-1205},
year = {2012},
doi = {10.1137/110825352}
}

@article{direct,
author={Jones, D. R.
and Perttunen, C. D.
and Stuckman, B. E.},
title={Lipschitzian optimization without the Lipschitz constant},
journal={J. Optim. Theory Appl.},
year={1993},
month={Oct},
volume={79},
pages={157-181},
number={1},
doi={10.1007/BF00941892}
}

@article{Ye12,
    author = {Ye, J. J.},
    title = {The exact penalty principle},
    journal = {Nonlinear Anal.},
    year = {2012},
    volume={75},
    number={3},
    pages={1642–1654},
    doi={10.1016/j.na.2011.03.025}
}

@article{JoTwWe92,
    author ={Jongen, H.T. and  Twilt, F. and Weber, G.W.} ,
    title = {Semi-infinite optimization: Structure and stability of the feasible set},
    journal = {J. Optim. Theory Appl.},
    volume={72},
    number={3},
    pages={529–552.},
    year = {1992},
    doi={10.1007/BF00939841}
}

@article{Io16,
    author = {Ioffe, A.D.},
    title = {Metric regularity-a survey part II. Applications.},
    journal = {J. Austral. Math. Soc.},
    volume={101},
    number={3},
    year ={2016},
    pages={376-417},
    doi={10.1017/S1446788715000695}
}

@article{He92,
    author = {Henrion, R.},
    title = {On constraint qualifications},
    journal = {J. Optim. Theory Appl.
},
    year = {1992},
    pages={187-197},
    volume={72},
    number={1},
    doi={10.1007/BF00939955
}
}

\end{document}